\documentclass{amsart}
\usepackage{amsmath, amssymb, amsthm, mathrsfs, amscd}
\usepackage{graphicx}

\theoremstyle{plain} 
 \newtheorem{thm}{Theorem}[section]
 \newtheorem{lem}[thm]{Lemma}
 \newtheorem{cor}[thm]{Corollary}
 \newtheorem{prop}[thm]{Proposition}
 \newtheorem{claim}[thm]{Claim}

\theoremstyle{definition}
  \newtheorem{defn}[thm]{Definition}

\theoremstyle{remark}
  \newtheorem{rem}[thm]{Remark}

\newcommand{\comm}{{\rm Comm}}
\renewcommand{\mod}{{\rm Mod}}
\renewcommand{\pmod}{{\rm PMod}}
\newcommand{\cal}{\mathcal}
\newcommand{\calc}{\mathcal{C}}
\newcommand{\calcp}{\mathcal{CP}}

\newcommand{\ci}[2]{\cite[#1]{#2}}

\begin{document}

\title{Commensurators of surface braid groups}
\author{Yoshikata Kida}
\author{Saeko Yamagata}
\address{Department of Mathematics, Kyoto University, 606-8502 Kyoto, Japan}
\email{kida@math.kyoto-u.ac.jp}
\address{Akashi National College of Technology, 674-8501 Akashi, Japan}
\email{yamagata@akashi.ac.jp}
\date{April 17, 2010, revised on February 20, 2011}

\begin{abstract}
We prove that if $g$ and $n$ are integers at least two, then the abstract commensurator of the braid group with $n$ strands on a closed orientable surface of genus $g$ is naturally isomorphic to the extended mapping class group of a compact orientable surface of genus $g$ with $n$ boundary components.
\end{abstract}

\maketitle


\section{Introduction}

For a positive integer $n$ and a manifold $M$, we define $PB_n(M)$ to be the {\it pure braid group} of $n$ strands on $M$, i.e., the fundamental group of the space of ordered distinct $n$ points in $M$.
Let $S$ be a connected, compact and orientable surface of genus $g$ with $p$ boundary components.
Let $\bar{S}$ denote the closed surface obtained by attaching disks to all boundary components of $S$.
We define $P(S)$ as the kernel of the homomorphism $\iota \colon \pmod(S)\rightarrow \mod(\bar{S})$ associated with the inclusion of $S$ into $\bar{S}$, where $\pmod(S)$ denotes the pure mapping class group for $S$, and $\mod(\bar{S})$ denotes the mapping class group for $\bar{S}$ (see Section \ref{subsec-group} for a definition of these groups).
As discussed in Section 4.1 of \cite{birman}, the Birman exact sequence tells us that if $g\geq 2$ and $p\geq 1$, then $PB_p(\bar{S})$ is identified with $P(S)$.

Automorphisms of $P(S)$ are studied in \cite{belli}, \cite{iim} and \cite{zhang}. 
The aim of this paper is to describe any isomorphism between finite index subgroups of $P(S)$. 
The same problem is also considered for the subgroup, denoted by $P_s(S)$, of $P(S)$ generated by all HBC twists and all HBP twists about separating HBPs in $S$.
We denote by $\mod^*(S)$ the {\it extended mapping class group} for $S$, i.e., the group of isotopy classes of homeomorphisms from $S$ onto itself, where isotopy may move points of the boundary of $S$.
Since $P(S)$ and $P_s(S)$ are normal subgroups of $\mod^*(S)$, the conjugation by each element of $\mod^*(S)$ defines an automorphism of $P(S)$ and of $P_s(S)$. 
The following theorem says that any isomorphism between finite index subgroups of $P(S)$ and of $P_s(S)$ can be described in this way.

\begin{thm}\label{thm-comm}
Let $S$ be a connected, compact and orientable surface of genus $g$ with $p$ boundary components.
We assume $g\geq 2$ and $p\geq 2$.
Then the following assertions hold:
\begin{enumerate}
\item Let $\Gamma_1$ and $\Gamma_2$ be finite index subgroups of $P(S)$, and let $f\colon \Gamma_1\rightarrow \Gamma_2$ be an isomorphism.
Then there exists an element $\gamma$ of $\mod^*(S)$ with $f(x)=\gamma x\gamma^{-1}$ for any $x\in \Gamma_1$.
\item Let $\Lambda_1$ and $\Lambda_2$ be finite index subgroups of $P_s(S)$, and let $h\colon \Lambda_1\rightarrow \Lambda_2$ be an isomorphism.
Then there exists an element $\lambda$ of $\mod^*(S)$ with $h(y)=\lambda y\lambda^{-1}$ for any $y\in \Lambda_1$.
\end{enumerate}
\end{thm}

For a group $\Gamma$, we define $F(\Gamma)$ to be the set of isomorphisms between finite index subgroups of $\Gamma$. 
We say that two elements $f$, $h$ of $F(\Gamma)$ are equivalent if there exists a finite index subgroup of $\Gamma$ on which $f$ and $h$ are equal. 
The composition of two elements $f\colon \Gamma_1\rightarrow \Gamma_2$, $h\colon \Lambda_1\rightarrow \Lambda_2$ of $F(\Gamma)$ given by $f\circ h\colon h^{-1}(\Gamma_1\cap \Lambda_2)\rightarrow f(\Lambda_2\cap \Gamma_1)$ induces the product operation on the quotient set of $F(\Gamma)$ by this equivalence relation. 
This makes it into the group called the {\it abstract commensurator} of $\Gamma$ and denoted by $\comm(\Gamma)$.

Since $P(S)$ and $P_s(S)$ are normal subgroups of $\mod^*(S)$, the homomorphisms
\[{\bf i}\colon \mod^*(S)\rightarrow \comm(P(S)),\quad {\bf i}_s\colon \mod^*(S)\rightarrow \comm(P_s(S))\]
are defined by conjugation.
Theorem \ref{thm-comm} shows that if $g\geq 2$ and $p\geq 2$, then ${\bf i}$ and ${\bf i}_s$ are surjective and thus isomorphisms by Lemma \ref{lem-faithful}.

For a positive integer $n$ and a manifold $M$, we define $B_n(M)$ as the {\it braid group} of $n$ strands on $M$, i.e., the fundamental group of the space of non-ordered distinct $n$ points in $M$.
The group $PB_n(M)$ is identified with a subgroup of $B_n(M)$ of index $n!$.
We note that if $\Gamma$ is a group and if $\Lambda$ is a finite index subgroup of $\Gamma$, then the natural homomorphism from $\comm(\Lambda)$ into $\comm(\Gamma)$ is an isomorphism.
We therefore obtain the following:

\begin{cor}
Let $g$ and $n$ be integers at least two.
Let $M$ be a connected, closed and orientable surface of genus $g$. 
Then $\comm(B_n(M))$ and $\comm(PB_n(M))$ are isomorphic to $\mod^*(S)$, where $S$ is a connected, compact and orientable surface of genus $g$ with $n$ boundary components.
\end{cor}

Let us mention surfaces excluded in Theorem \ref{thm-comm}. 
\begin{itemize}
\item If $p=0$, then both $P(S)$ and $P_s(S)$ are trivial.
\item If $g\geq 2$ and $p=1$, then $P(S)$ is isomorphic to $\pi_1(\bar{S})$ by the Birman exact sequence, and $P_s(S)$ is identified with a subgroup of the commutator subgroup of $\pi_1(\bar{S})$ and thus isomorphic to a non-abelian free group.
\item If $g=0$, then we have $P(S)=P_s(S)=\pmod(S)$.
\item If $g=1$, then we have $P(S)=\cal{I}(S)$ and $P_s(S)=\cal{K}(S)$, where $\cal{I}(S)$ is the Torelli group for $S$, and $\cal{K}(S)$ is the Johnson kernel for $S$ (see \cite{kida-tor} for a definition of these groups).
\end{itemize}
When $g=0, 1$, a description of any isomorphism between finite index subgroups of $P(S)$ and of $P_s(S)$ is therefore given in \cite{kork-aut} and \cite{kida-tor}, respectively.

As studied in \cite{iva-aut}, \cite{kork-aut} and \cite{luo}, the complex of curves for $S$ plays an important role in the computation of the abstract commensurator of $\mod^*(S)$. 
Afterward, in \cite{bm}, \cite{bm-add}, \cite{farb-ivanov}, \cite{kida-tor} and \cite{mv}, automorphisms and the abstract commensurators of certain subgroups of $\mod^*(S)$ are understood through the study of appropriate variants of the complex of curves. 
To prove Theorem \ref{thm-comm}, we follow this strategy and introduce two simplicial complexes $\calcp(S)$ and $\calcp_s(S)$ associated to $S$, inspired by a work due to Irmak-Ivanov-McCarthy \cite{iim}. 
Vertices of those simplicial complexes are defined as isotopy classes of certain simple closed curves in $S$ and pairs of them.
Simplices are defined in terms of disjointness of curves in the same manner as the definition of the complex of curves. 
We then have the natural action of $\mod^*(S)$ on $\calcp(S)$ and on $\calcp_s(S)$. Theorem \ref{thm-comm} can be deduced by combining the following two assertions: If $g\geq 2$ and $p\geq 2$, then
\begin{itemize}
\item[(1)] any isomorphism between finite index subgroups of $P(S)$ (resp.\ $P_s(S)$) induces an automorphism of $\calcp(S)$ (resp.\ $\calcp_s(S)$); and
\item[(2)] any automorphism of $\calcp(S)$ (resp.\ $\calcp_s(S)$) is induced by an element of $\mod^*(S)$.
\end{itemize}
Assertion (1) is proved in Theorem \ref{thm-induce}, and assertion (2) is proved in Corollary \ref{cor-aut}. 
In the proof of assertion (1), we present an algebraic characterization of certain elements of $P(S)$ associated to vertices of $\calcp(S)$, based on \cite{iim}.

This paper is organized as follows.
In Section \ref{sec-pre}, we provide basic terminology and the definition of simplicial complexes and groups discussed above. 
In the final subsection of Section \ref{sec-pre}, we present an outline of the proof of assertion (2) given throughout Sections \ref{sec-basic}--\ref{supseinj_Phi}. 
In Section \ref{sec-cha}, we prove assertion (1). 
Moreover, we show that any injective homomorphism from a finite index subgroup of $P_s(S)$ into $P(S)$ induces a superinjective map from $\calcp_s(S)$ into $\calcp(S)$.
This result is used in our subsequent paper \cite{ky-cohop}, where we conclude that any injective homomorphism from a finite index subgroup of $P(S)$ into $P(S)$ is the conjugation by an element of $\mod^*(S)$ if $g\geq 2$ and $p\geq 2$.
The same conclusion is also proved for $P_s(S)$.


\section{Preliminaries}\label{sec-pre}

\subsection{Notation and terminology}\label{subsec-term}

Unless otherwise stated, we assume a surface to be connected, compact and orientable. 
Let $S=S_{g, p}$ be a surface of genus $g$ with $p$ boundary components. 
A simple closed curve in $S$ is said to be {\it essential} in $S$ if it is neither homotopic to a single point of $S$ nor isotopic to a component of $\partial S$.
When there is no confusion, we mean by a curve in $S$ either an essential simple closed curve in $S$ or the isotopy class of it. 
A curve $\alpha$ in $S$ is said to be {\it separating} in $S$ if $S\setminus \alpha$ is not connected.
Otherwise $\alpha$ is said to be {\it non-separating} in $S$.
These properties depend only on the isotopy class of $\alpha$.
We mean by a {\it holed sphere} a surface of genus zero with non-empty boundary.

\medskip

\noindent {\bf Hole-bounding curves (HBC).} A curve $\alpha$ in $S$ is called a {\it hole-bounding curve (HBC)} in $S$ if $\alpha$ is separating in $S$ and cuts off a holed sphere from $S$. 
When $g\geq 1$, if the holed sphere cut off by $\alpha$ from $S$ contains exactly $k$ components of $\partial S$, then we call $\alpha$ a {\it $k$-HBC} in $S$.
Note that we have $2\leq k\leq p$.

\medskip

\noindent {\bf Hole-bounding pairs (HBP).} A pair $\{ \alpha, \beta \}$ of curves in $S$ is called a {\it hole-bounding pair (HBP)} in $S$ if
\begin{itemize}
\item $\alpha$ and $\beta$ are disjoint and non-isotopic;
\item either $\alpha$ and $\beta$ are non-separating in $S$ or $\alpha$ and $\beta$ are separating in $S$ and are not an HBC in $S$; and
\item $S\setminus (\alpha \cup \beta)$ is not connected and has a component of genus zero.
\end{itemize}
We note that if $g\geq 2$, then the component of genus zero in the last condition, denoted by $Q$, uniquely exists. 
In this case, if $Q$ contains exactly $k$ components of $\partial S$, then we call the pair $\{ \alpha, \beta \}$ a {\it $k$-HBP} in $S$. 
Note that we have $1\leq k\leq p$.
An HBP in $S$ is said to be {\it non-separating} in $S$ if both curves in it are non-separating in $S$.
Otherwise it is said to be {\it separating} in $S$ (see Figure \ref{fig-hbchbp}). 

\begin{figure}
\begin{center}
\includegraphics[width=8cm]{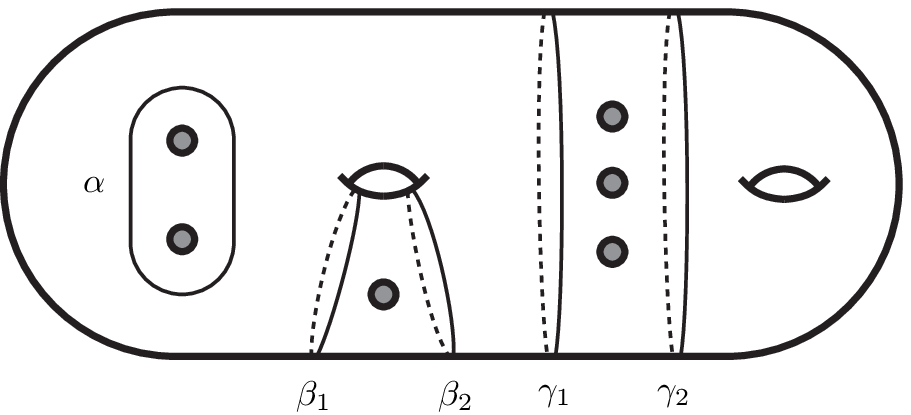}
\caption{$\alpha$ is an HBC, $\{ \beta_1, \beta_2\}$ is a non-separating HBP, and $\{ \gamma_1, \gamma_2\}$ is a separating HBP.}\label{fig-hbchbp}
\end{center}
\end{figure}

We define $V(S)$ as the set of isotopy classes of essential simple closed curves in $S$. 
We denote by $i\colon V(S)\times V(S)\rightarrow \mathbb{Z}_{\geq 0}$ the {\it geometric intersection number}, i.e., the minimal cardinality of the intersection of representatives for two elements of $V(S)$.
Let $\Sigma(S)$ denote the set of non-empty finite subsets $\sigma$ of $V(S)$ with $i(\alpha, \beta)=0$ for any $\alpha, \beta \in \sigma$.
For an element $\sigma$ of $\Sigma(S)$, we mean by a {\it representative} of $\sigma$ the union of mutually disjoint representatives of elements of $\sigma$.
We extend the function $i$ to the symmetric function on the square of $V(S)\sqcup \Sigma(S)$ with $i(\alpha, \sigma)=\sum_{\beta \in \sigma}i(\alpha, \beta)$ and $i(\sigma, \tau)=\sum_{\beta \in \sigma, \gamma \in \tau}i(\beta, \gamma)$ for any $\alpha \in V(S)$ and $\sigma, \tau \in \Sigma(S)$.
We say that two elements $\alpha$, $\beta$ of $V(S)\sqcup \Sigma(S)$ are {\it disjoint} if $i(\alpha, \beta)=0$. 
Otherwise, we say that $\alpha$ and $\beta$ {\it intersect}.
We say that two elements $\alpha$, $\beta$ of $V(S)$ {\it fill} $S$ if there exists no element of $V(S)$ disjoint from both $\alpha$ and $\beta$.

Let $\sigma$ be an element of $\Sigma(S)$.
We denote by $S_{\sigma}$ the surface obtained by cutting $S$ along all curves in $\sigma$.
When $\sigma$ consists of a single curve $\alpha$, we denote it by $S_{\alpha}$.
Each component of $S_{\sigma}$ is often identified with a complementary component of a tubular neighborhood of a one-dimensional submanifold representing $\sigma$ in $S$.
Let $R$ be a component of $S_{\sigma}$.
The set $V(R)$ is then identified with a subset of $V(S)$.


\subsection{Simplicial complexes associated to a surface}\label{subsec-complex}

Let $S$ be a surface.
We recall two complexes of curves and then introduce complexes of HBCs and HBPs.

\medskip

\noindent {\bf Complex $\calc(S)$.} We define $\calc(S)$ as the abstract simplicial complex such that the sets of vertices and simplices of $\calc(S)$ are $V(S)$ and $\Sigma(S)$, respectively, and call it the {\it complex of curves} for $S$.

\medskip

This complex was introduced by Harvey \cite{harvey}.
The following theorem says that for almost all surfaces $S$, any automorphism of $\calc(S)$ is generally induced by an element of $\mod^*(S)$.
This fact is fundamental in the study of the abstract commensurators of various subgroups of the mapping class group.

\begin{thm}[\cite{iva-aut}, \cite{kork-aut}, \cite{luo}]\label{thm-cc}
Let $S=S_{g, p}$ be a surface with $3g+p-4>0$.
Then the following assertions hold:
\begin{enumerate}
\item If $(g, p)\neq (1, 2)$, then any automorphism of $\calc(S)$ is induced by an element of $\mod^*(S)$.
\item If $(g, p)=(1, 2)$, then any automorphism of $\calc(S)$ preserving vertices which correspond to separating curves in $S$ is induced by an element of $\mod^*(S)$.
\end{enumerate}
\end{thm}

Let $\bar{S}$ be the closed surface obtained from $S$ by attaching disks to all boundary components of $S$.
Let $\calc^*(\bar{S})$ be the simplicial cone over $\calc(\bar{S})$ with its cone point $\ast$.
Namely, $\calc^*(\bar{S})$ is the abstract simplicial complex such that the set of vertices of $\calc^*(\bar{S})$ is the disjoint union $V(\bar{S})\sqcup \{ \ast \}$; and the set of simplices of $\calc^*(\bar{S})$ is
\[\Sigma(\bar{S})\cup \{\, \sigma \cup \{ \ast \} \mid \sigma \in \Sigma(\bar{S})\cup \{ \emptyset \}\,\}.\]
We have the simplicial map $\pi \colon \calc(S)\rightarrow \calc^*(\bar{S})$ associated with the inclusion of $S$ into $\bar{S}$.
Note that $\pi^{-1}(\{ \ast \})$ consists of all HBCs in $S$.
An HBP in $S$ is identified with an edge of $\calc(S)$.
For each edge $\{ \alpha, \beta \}$ of $\calc(S)$, it is an HBP in $S$ if and only if we have $\pi(\alpha)=\pi(\beta)\neq \ast$.
Two disjoint HBPs $a$, $b$ in $S$ are said to be {\it equivalent} in $S$ if $\pi(a)=\pi(b)$.

\medskip

\noindent {\bf Complex $\calc_s(S)$.} Let $V_s(S)$ denote the set of all elements of $V(S)$ whose representatives are separating in $S$.
We define $\calc_s(S)$ as the full subcomplex of $\calc(S)$ spanned by $V_s(S)$ and call it the {\it complex of separating curves} for $S$.

\medskip

This complex (for closed surfaces) appears in \cite{bm}, \cite{bm-add}, \cite{farb-ivanov} and \cite{mv}.
Automorphisms of $\calc_s(S)$ are studied in \cite{bm}, \cite{bm-add} and \cite{kida-tor}.
We now introduce two simplicial complexes whose vertices are HBCs and HBPs in $S$, inspired by the work due to Irmak-Ivanov-McCarthy \cite{iim} characterizing elements of $P(S)$ associated with HBCs and HBPs in $S$ algebraically.

\medskip

\noindent {\bf Complexes $\calcp(S)$ and $\calcp_s(S)$.} Let $V_c(S)$ denote the set of all elements of $V(S)$ whose representatives are HBCs in $S$.
Let $V_p(S)$ denote the set of all elements of $\Sigma(S)$ whose representatives are HBPs in $S$.
 
We define $\calcp(S)$ as the abstract simplicial complex such that the set of vertices is the disjoint union $V_c(S)\sqcup V_p(S)$, and a non-empty finite subset $\sigma$ of $V_c(S)\sqcup V_p(S)$ is a simplex of $\calcp(S)$ if and only if any two elements of $\sigma$ are disjoint. 
We call elements of $V_c(S)$ and $V_p(S)$ {\it HBC-vertices} and {\it HBP-vertices} of $\calcp(S)$, respectively.

We define $\calcp_s(S)$ as the full subcomplex of $\calcp(S)$ spanned by all vertices whose representatives are either an HBC in $S$ and or a separating HBP in $S$.

\medskip

We note that if $S$ is of genus zero, then $\calcp(S)=\calcp_s(S)=\calc(S)=\calc_s(S)$.
If $S$ is of genus one, then $\calcp(S)$ is equal to the Torelli complex studied in \cite{bm} and \cite{kida-tor}, and the equality $\calcp_s(S)=\calc_s(S)$ holds.

\medskip

\noindent {\bf Superinjective maps.} Let $X$ and $Y$ be any of the four simplicial complexes introduced above. 
We denote by $V(X)$ and $V(Y)$ the sets of vertices of $X$ and $Y$, respectively. 
Note that a map $\phi \colon V(X)\rightarrow V(Y)$ defines a simplicial map from $X$ into $Y$ if and only if $i(\phi(a), \phi(b))=0$ for any $a, b\in V(X)$ with $i(a, b)=0$. 
We mean by a {\it superinjective map} $\phi \colon X\rightarrow Y$ a simplicial map $\phi \colon X\rightarrow Y$ satisfying $i(\phi(a), \phi(b))\neq 0$ for any $a, b\in V(X)$ with $i(a, b)\neq 0$. 
One can check that any superinjective map from $X$ into $Y$ is injective, along the proof of Lemma 3.1 in \cite{irmak1} proving that any superinjective map from $\calc(S)$ into itself is injective.


\subsection{Surface braid groups}\label{subsec-group}

Let $S$ be a surface.
The {\it mapping class group} $\mod(S)$ for $S$ is defined as the subgroup of $\mod^*(S)$ consisting of all isotopy classes of orientation-preserving homeomorphisms from $S$ onto itself. 
The {\it pure mapping class group} $\pmod(S)$ for $S$ is defined as the subgroup of $\mod^*(S)$ consisting of all isotopy classes of homeomorphisms from $S$ onto itself that preserve an orientation of $S$ and each component of $\partial S$ as a set. 
We refer to \cite{iva-mcg} for fundamentals of these groups.

For each $\alpha \in V(S)$, let $t_{\alpha}\in \pmod(S)$ denote the {\it (left) Dehn twist} about $\alpha$.
The Dehn twist about an HBC is called an {\it HBC twist}.
For each HBP $\{ \alpha, \beta \}$ in $S$, the elements $t_{\alpha}t_{\beta}^{-1}$, $t_{\beta}t_{\alpha}^{-1}$ of $\pmod(S)$ are called {\it HBP twists} about the HBP $\{ \alpha, \beta \}$.

Let $\bar{S}$ be the closed surface obtained by attaching disks to all boundary components of $S$. 
We then have the surjective homomorphism
\[\iota \colon \pmod(S)\rightarrow \mod(\bar{S})\]
associated with the inclusion of $S$ into $\bar{S}$.
We define $P(S)$ to be $\ker \iota$. 
The group $P(S)$ is known to be generated by all HBC twists and all HBP twists in $S$ (see Section 4.1 in \cite{birman}). 
We define $P_s(S)$ to be the subgroup of $P(S)$ generated by all HBC twists and all HBP twists about separating HBPs in $S$.

\begin{lem}\label{lem-faithful}
Let $S=S_{g, p}$ be a surface with $g\geq 2$ and $p\geq 1$.
Then the following assertions hold:
\begin{enumerate}
\item The actions of $\mod^*(S)$ on $\calcp(S)$ and on $\calcp_s(S)$ are faithful.
\item The homomorphisms
\[{\bf i}\colon \mod^*(S)\rightarrow \comm(P(S)),\quad {\bf i}_s\colon \mod^*(S)\rightarrow \comm(P_s(S))\]
defined by conjugation are injective.
\end{enumerate}
\end{lem}

\begin{proof}
Let $x$ be an element of $\mod^*(S)$ fixing any vertex of $\calcp_s(S)$.
Pick $\alpha \in V_s(S)\setminus V_c(S)$.
We choose separating HBPs $\{ \alpha, \beta_1 \}$, $\{ \alpha, \beta_2 \}$ in $S$ with $\beta_1\neq \beta_2$.
Since $x$ fixes these HBPs, it fixes $\alpha$.
Thus, $x$ fixes any element of $V_s(S)$.
For each non-separating curve $\gamma$ in $S$, we choose separating curves $\delta_1$, $\delta_2$ in $S$ disjoint from $\gamma$ and filling $S_{\gamma}$.
Since $x$ fixes $\delta_1$ and $\delta_2$, it fixes $\gamma$.
It follows that $x$ fixes any element of $V(S)$ and is thus neutral.
Assertion (i) is proved.

To prove assertion (ii), it suffices to show that ${\bf i}_s$ is injective.
Pick $y\in \mod^*(S)$ with ${\bf i}_s(y)$ neutral.
There exists a finite index subgroup of $P_s(S)$ such that $y$ commutes any element of it.
For any separating HBP $\{ \alpha, \beta \}$ in $S$, we then have a non-zero integer $n$ with $y(t_{\alpha}t_{\beta}^{-1})^ny^{-1}=(t_{\alpha}t_{\beta}^{-1})^n$.
Thus, $y$ fixes $\{ \alpha, \beta \}$.
Similarly, $y$ fixes any HBC in $S$.
By assertion (i), $y$ is neutral.
Assertion (ii) is proved.
\end{proof}

\begin{lem}\label{lem-multitwist}
Let $S$ be a surface of genus at least one.
Pick $\sigma \in \Sigma(S)$ and let $D_{\sigma}$ be the subgroup of $\pmod(S)$ generated by all Dehn twists about curves in $\sigma$.
Then the following assertions hold:
\begin{enumerate}
\item $D_{\sigma}\cap P(S)$ is generated by all Dehn twists about HBCs in $\sigma$ and all HBP twists about HBPs of two curves in $\sigma$.
\item $D_{\sigma}\cap P_s(S)$ is generated by all Dehn twists about HBCs in $\sigma$ and all HBP twists about separating HBPs of two curves in $\sigma$.
\end{enumerate}
In particular, $P_s(S)$ contains no non-zero power of the HBP twist about a non-separating HBP in $S$.
\end{lem}

\begin{proof}
Assertion (i) holds because any element of $P(S)$ induces the neutral element of $\mod(\bar{S})$.
We define $\cal{K}(S)$ to be the group generated by all Dehn twists about separating curves in $S$, and call it the {\it Johnson kernel} for $S$.
The group $P_s(S)$ is contained in $\cal{K}(S)$.
Theorem 6.1 (ii) in \cite{kida-tor} shows that $D_{\sigma}\cap \cal{K}(S)$ is generated by Dehn twists about curves in $\sigma \cap V_s(S)$.
Assertion (ii) thus follows because any element of $P_s(S)$ induces the neutral element of $\mod(\bar{S})$.
\end{proof}


\subsection{Plan of Sections \ref{sec-basic}--\ref{supseinj_Phi}}\label{subsec-plan}

Let $S=S_{g, p}$ be a surface with $g\geq 2$ and $p\geq 2$.
We explain an outline to prove that any automorphism $\phi$ of $\calcp(S)$ is induced by an element of $\mod^*(S)$.
This conclusion will be obtained by constructing an automorphism $\Phi$ of $\calc(S)$ inducing $\phi$.

In Section \ref{sec-basic}, we prove that
\begin{itemize}
\item $\phi$ preserves vertices corresponding to HBCs in $S$, non-separating HBPs in $S$ and separating HBPs in $S$, respectively; and
\item for any two disjoint HBPs $b_1$, $b_2$ in $S$ containing a common curve, the two HBPs $\phi(b_1)$, $\phi(b_2)$ also contain a common curve.
\end{itemize}
For each $\alpha \in V(S)$, we define $\Phi(\alpha)\in V(S)$ as follows.
If $\alpha$ is an HBC in $S$, then we put $\Phi(\alpha)=\phi(\alpha)$.
Otherwise, choosing two disjoint and distinct HBPs in $S$, say $b_1$ and $b_2$, containing $\alpha$, we define $\Phi(\alpha)$ to be the common curve of the two HBPs $\phi(b_1)$ and $\phi(b_2)$.
Sections \ref{sec-const} and \ref{sec-const-p3} are devoted to showing that this is well-defined.
In Section \ref{supseinj_Phi}, the map $\Phi$ is shown to be an automorphism of $\calc(S)$ and is thus induced by an element of $\mod^*(S)$ by Theorem \ref{thm-cc}.
In a similar way, for any automorphism $\psi$ of $\calcp_s(S)$, we construct a map $\Psi$ from $V_s(S)$ into itself and show it to be an automorphism of $\calc_s(S)$ and thus induced by an element of $\mod^*(S)$ by Theorem 1.2 in \cite{kida-tor}.


\section{Basic properties of automorphisms of $\calcp(S)$}\label{sec-basic}

For an automorphism $\phi$ of $\calcp(S)$, we show that $\phi$ preserves topological properties of HBCs and HBPs corresponding to vertices of $\calcp(S)$. 
In Section \ref{subsec-max}, simplices of $\calcp(S)$ of maximal dimension are completely described. It is obvious that such simplices are preserved by $\phi$.
In Section \ref{subsec-top}, using this fact, we show that $\phi$ preserves HBC-vertices and HBP-vertices, respectively, and preserves more detailed information on these vertices.

\subsection{Simplices of $\calcp(S)$ of maximal dimension}\label{subsec-max}

We say that two disjoint curves in $S$ are {\it HBP-equivalent} in $S$ if they either are equal or form an HBP in $S$. 
This defines an equivalence relation on each simplex $\sigma$ of $\calc(S)$ with $\sigma \cap V_c(S)=\emptyset$. 
For each simplex $\sigma$ of $\calc(S)$, we mean by an {\it HBP-equivalence class} in $\sigma$ an equivalence class in $\sigma \setminus V_c(S)$ with respect to this equivalence relation. 
Before describing simplices of $\calcp(S)$ of maximal dimension, we give several elementary observations on HBP-equivalence.

\begin{lem}\label{lem-hbp-dec}
Let $S=S_{g,p}$ be a surface with $g\geq2$ and $p\geq1$, and let $b\in \Sigma(S)$ be a simplex such that $|b|\geq 2$, $b\cap V_c(S)=\emptyset$ and any two curves of $b$ are HBP-equivalent in $S$.
Then the following assertions hold:
\begin{enumerate}
\item Either all curves of $b$ are non-separating in $S$ or they are separating in $S$.
\item If all curves of $b$ are non-separating in $S$, then each component of $S_b$ contains exactly two curves of $b$ as boundary components, and exactly one component of $S_b$ is of positive genus.
\item If all curves of $b$ are separating in $S$, then there exist exactly two components of $S_b$ of positive genus.
Moreover, those two components contain exactly one curve of $b$ as a boundary component, and any other component of $S_b$ contains exactly two curves of $b$ as boundary components.
\end{enumerate}
\end{lem}

\begin{proof}
Assertion (i) follows from the definition of HBPs.
Assertions (ii) and (iii) are verified by induction on $|b|$.
\end{proof}

\begin{lem}\label{comp}
Let $S=S_{g,p}$ be a surface with $g\geq2$ and $p\geq1$, and let $b, c\in \Sigma(S)$ be simplices such that 
\begin{itemize}
\item $|b| \geq 2$, $b\cap c=\emptyset$ and $i(b, c)=0$;
\item $b$ consists of non-separating curves in $S$; and
\item each of $b$ and $c$ is an HBP-equivalence class in the simplex $b\cup c$.
\end{itemize}
Then any curve of $c$ is contained in the component of $S_b$ of positive genus.
\end{lem}

\begin{proof}
This lemma follows from the fact that any curve in a component of $S_b$ of genus zero is either an HBC in $S$ or HBP-equivalent to a curve in $b$.
\end{proof}

\begin{lem}\label{comp-s}
Let $S=S_{g,p}$ be a surface with $g\geq2$ and $p\geq1$.
Pick a separating curve $\alpha$ in $S$ which is not an HBC in $S$, and pick a simplex $c\in \Sigma(S)$ such that 
\begin{itemize}
\item $|c| \geq 2$, $\alpha\not\in c$ and $i(\alpha, c)=0$; and
\item $c$ is an HBP-equivalence class in the simplex $\{ \alpha \}\cup c$.
\end{itemize}
Then there exists a component $Q$ of $S_{\alpha}$ such that $c$ belongs to $\Sigma(Q)$.
\end{lem}

\begin{proof}
Pick two distinct curves $\gamma$, $\delta$ of $c$.
Let $Q$ and $R$ be the components of $S_{\alpha}$ with $\gamma \in V(Q)$ and $\delta \in V(R)$. 
If all curves of $c$ are non-separating in $S$, then $\gamma$ and $\delta$ are non-separating in $Q$ and $R$, respectively. 
Since $S_{\{\gamma,\delta\}}$ is not connected, we have $Q=R$.
If all curves of $c$ are separating in $S$, then for each curve $\beta$ of $c$, since $\beta$ is not an HBC in $S$ and not equivalent to $\alpha$, the curve $\beta$ has to separate the component of $S_{\alpha}$ containing $\beta$ into two components of positive genus. 
Since any two curves in $c$ are HBP-equivalent, all curves of $c$ are contained in the same component of $S_{\alpha}$.
\end{proof}

\begin{lem}\label{number_HBC}
Let $S=S_{g,p}$ be a surface with $g \geq 2$ and $p \geq 2$, and let $\sigma$ be a simplex of $\calcp(S)$ consisting of HBC-vertices. 
Then the inequality $|\sigma| \leq p-1$ holds, and the equality can be attained. 
\end{lem}

\begin{proof}
We can find a simplex $\tau$ of $\calcp(S)$ consisting of $p-1$ HBC-vertices.
The inequality in the lemma is verified by induction on $p$.
\end{proof}

We now describe simplices of $\calcp(S)$ of maximal dimension in the following:

\begin{prop}\label{max_dim}
Let $S=S_{g,p}$ be a surface with $g \geq 2$ and $p\geq 1$. 
Then we have
\[{\rm dim}(\calcp(S))=\binom{p+1}{2}-1.\]
Moreover, for any simplex $\sigma$ of $\calcp(S)$ of maximal dimension, 
there exists a unique simplex $s=\{\beta_1,\beta_2,\ldots,\beta_{p+1}\}$ of $\calc(S)$ such that
\begin{itemize}
\item[(a)] any two curves in $s$ are HBP-equivalent; and 
\item[(b)] $\sigma$ consists of all HBPs of two curves in $s$.
\end{itemize}
\end{prop}

Before proving this proposition, we note that if $S$ is a surface of genus one, then $\calcp(S)$ is equal to the Torelli complex $\cal{T}(S)$ studied in \cite{kida-tor}, where simplices of $\cal{T}(S)$ of maximal dimension are described.
\begin{figure}
\begin{center}
\includegraphics[width=12cm]{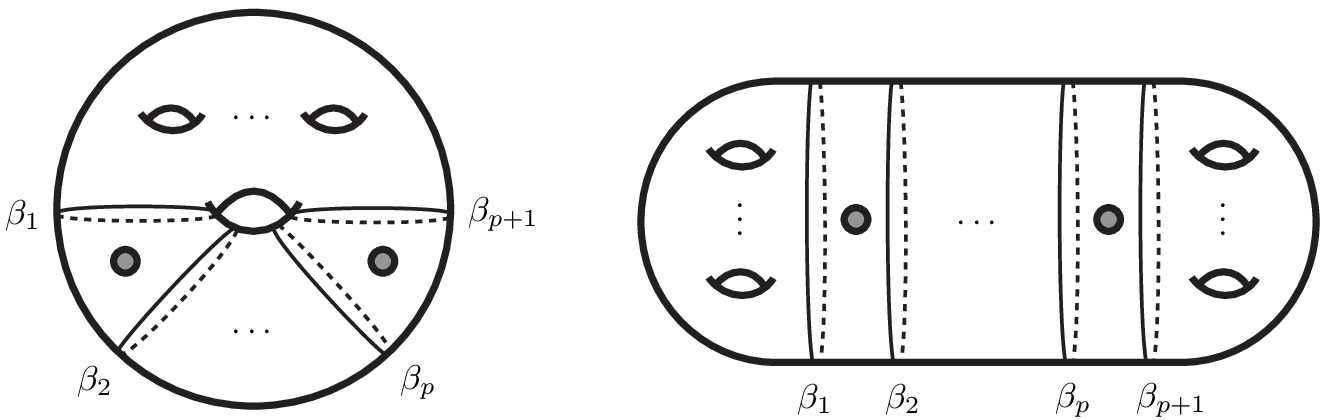}
\end{center}
\caption{Simplices of $\calcp(S)$ of maximal dimension}\label{fig-max}
\end{figure}
Proposition 3.4 in \cite{kida-tor} thus implies that if $S=S_{1, p}$ is a surface with $p\geq 2$, then we have
\[{\rm dim}(\calcp(S))=\binom{p}{2}-1.\]

\begin{proof}[Proof of Proposition \ref{max_dim}]
The $p+1$ curves in $S$ described in Figure \ref{fig-max} are mutually HBP-equivalent. 
It follows that the dimension of $\calcp(S)$ is not smaller than the right hand side of the equality in the proposition.

In what follows, we show the equality by induction on $p$. 
The equality obviously holds if $p=1$. 
We assume $p \geq 2$.
Pick a simplex $\sigma$ of $\calcp(S)$ of maximal dimension. 
We define the simplex of $\calc(S)$, denoted by
\[s=\{\alpha_1,\ldots,\alpha_k,\beta_{1 1},\ldots,\beta_{1 m_1},\beta_{2 1},\ldots,\beta_{l m_l}\},\]
so that
\begin{itemize}
\item $\{ \alpha_1,\ldots,\alpha_k\}$ is the collection of HBCs of $\sigma$;
\item $\{ \beta_{1 1},\ldots,\beta_{1 m_1},\beta_{2 1},\ldots,\beta_{l m_l}\}$ is the collection of curves in HBPs of $\sigma$; and 
\item for each $j=1,\ldots,l$, the set $b_j=\{\beta_{j 1},\ldots,\beta_{j m_j}\}$ is an HBP-equivalence class in $s$. 
\end{itemize} 
Since $\dim \sigma$ is maximal, $\sigma$ contains all HBPs of two curves in each $b_j$. 
We hence obtain the equality
\[|\sigma|=k+\sum_{j=1}^l\binom{m_j}{2}.\]

\begin{claim}\label{k=0}
We have $k=0$.
\end{claim}

\begin{proof}
We assume $k \geq 1$. 
For each $j=1,2,\ldots,k$, let $n_j$ be the integer with $\alpha_j$ an $n_j$-HBC in $S$, and put $n=\max\{n_1,n_2,\ldots,n_k\}$.
Exchanging indices if necessary, we may suppose that $\alpha_1$ is an $n$-HBC in $S$. 
Let $\partial_1,\ldots,\partial_n$ be the boundary components of $S$ contained in the component of $S_{\alpha_1}$ of genus zero. 
The maximality of $\dim \sigma$ and $n$ implies that there exists an $n$-HBP $b=\{\gamma_1, \gamma_2\}$ in $\sigma$ such that the component of $S_b$ of genus zero, denoted by $Q$, contains $\partial_1,\ldots,\partial_n$.
We note that $\sigma \cap V(Q)$ consists of HBCs in $S$ and that the inequality $|\sigma \cap V(Q)|\leq n-1$ holds.
Remove all curves in $\sigma \cap V(Q)$ from $\sigma$ and add curves $\gamma_3,\gamma_4\ldots,\gamma_{n+1} \in V(Q)$ to $\sigma$ such that for each $j=3,4,\ldots,n+1$, $\gamma_j$ is HBP-equivalent to $\gamma_1$ and $\gamma_2$ as a curve in $S$. 
This new collection of curves is denoted by $s'\in \Sigma(S)$ and associates the simplex $\sigma'$ of $\calcp(S)$ consisting of all HBCs in $s'$ and all HBPs of two curves in each HBP-equivalence class in $s'$.
We thus obtain the inequality  
\[|\sigma'| \geq |\sigma| - |\sigma \cap V(Q)|+\frac{\, (n+1)n\,}{2}-1\geq |\sigma|+\frac{\,n(n-1)\,}{2}> |\sigma|,\]
where the last inequality holds since we have $n\geq2$.
This contradicts the maximality of $\dim \sigma$. 
\end{proof}

\begin{claim}\label{l=1}
We have $l=1$.
\end{claim}

\begin{proof}
Assume $l\geq2$. 
We deduce a contradiction in the following two cases: (i) $b_1$ consists of separating curves in $S$; and (ii) $b_1$ consists of non-separating curves in $S$. 
Let $R$ be the component of $S_{b_1}$ containing all curves in $b_2$, which exists by Lemmas \ref{comp} and \ref{comp-s}. 
Let $p_1$ denote the number of boundary components of $S$ contained in $R$. Since $b_1$ and $b_2$ are not equivalent, the inequality $1\leq p_1\leq p-1$ holds, and the genus of $R$ is positive.

In case (i), we may assume that $\beta_{11}$ is the element of $b_1$ corresponding to a component of $\partial R$.
Moreover, after exchanging indices, we may assume that there exists an integer $l'$ with $2\leq l'\leq l$ such that
\begin{itemize}
\item for each $j=2,\ldots, l'$, all curves in $b_j$ are contained in $R$; and
\item for each $j=l'+1,\ldots, l$, all curves in $b_j$ are contained in a component of $S_{b_1}$ distinct from $R$.
\end{itemize}

We mean by a {\it handle} a surface homeomorphic to $S_{1, 1}$.
Let $R'$ denote the surface obtained from $R$ by attaching a handle to the component of $\partial R$ corresponding to $\beta_{11}$. 
The number of boundary components of $R'$ is then equal to $p_1$. 
For each $j=2,\ldots, l'$, any curve of $b_j$ can be regarded as a curve in $R'$ via the inclusion of $R$ into $R'$, and any two curves of $b_j$ form an HBP in $R'$. 
We denote by $\sigma_{R'}$ the simplex of $\calcp(R')$ consisting of all HBPs in $R'$ associated with two curves in $b_j$ for $j=2,\ldots, l'$. 
Since the genus of $R'$ is at least two, we obtain
\[|\sigma_{R'}|=\sum_{j=2}^{l'}\binom{m_j}{2} \leq \binom{p_1+1}{2}\]
by the hypothesis of the induction.
We remove all curves in $b_2, b_3,\ldots, b_{l'}$ from $s$ and add to $s$ curves $\delta_1,\delta_2,\ldots,\delta_{p_1} \in V(R)$ which are mutually disjoint and HBP-equivalent to $\beta_{11}$ as curves in $S$. 
This new collection of curves is denoted by $s_1\in \Sigma(S)$. 
Let $\sigma_1$ be the simplex of $\calcp(S)$ consisting of all HBPs of two curves in HBP-equivalence classes of $s_1$. 
We then obtain the equality
\[|\sigma_1|=\binom{m_1+p_1}{2}+\sum_{j=l'+1}^l\binom{m_j}{2}\] 
and the inequality
\begin{align*}
|\sigma_1|-|\sigma|&=\binom{m_1+p_1}{2}+\sum_{j=l'+1}^l\binom{m_j}{2}
                                 -\sum_{j=1}^l \binom{m_j}{2}\\
                            &\geq \binom{m_1+p_1}{2}-\binom{p_1+1}{2}-\binom{m_1}{2}=(m_1-1)p_1>0.                  
\end{align*}
This contradicts the maximality of $\dim \sigma$.

Let us consider case (ii).
Without loss of generality, we may assume that $\beta_{11}$ and $\beta_{12}$ are the two elements of $b_1$ corresponding to components of $\partial R$. 
It follows from Lemma \ref{comp} that for each $j=2,\ldots, l$, any curve of $b_j$ is contained in $R$ as an essential one.
Let $R''$ be the surface obtained by identifying $\beta_{11}$ with $\beta_{12}$, whose genus is at least two.
For each $j=2,\ldots, l$, any curve of $b_j$ can be regarded as a curve in $R''$ via the natural map from $R$ into $R''$, and any two curves of $b_j$ form an HBP in $R''$. We denote by $\sigma_{R''}$ the simplex of $\calcp(R'')$ consisting of all HBPs in $R''$ associated with two curves in $b_j$ for $j=2,\ldots, l$. 
Along an argument of the same kind as in case (i), we can deduce a contradiction. 
\end{proof}

Since any HBP-equivalence class in a simplex of $\calc(S)$ consists of at most $p+1$ curves, Claims \ref{k=0} and \ref{l=1} imply the equality in the proposition. 
Existence and uniqueness of the simplex $s=\{\beta_1,\beta_2,\ldots,\beta_{p+1}\}$ in the proposition also follow.  
\end{proof}


\subsection{Topological properties preserved by $\phi$}\label{subsec-top}

The following three lemmas are obtained as a consequence of Proposition \ref{max_dim}.

\begin{lem}\label{HBP_to_HBP}
Let $S=S_{g,p}$ be a surface with $g \geq 2$ and $p \geq 2$, and let $\phi \colon \calcp(S)\rightarrow \calcp(S)$ be a superinjective map.
Then the following assertions hold:
\begin{enumerate}
\item The map $\phi$ preserves HBPs in $S$.
\item If $b_1$ and $b_2$ are disjoint and equivalent HBPs in $S$, then $\phi(b_1)$ and $\phi(b_2)$ are also equivalent.
\end{enumerate}
\end{lem}

\begin{proof}
Let $b$ be an HBP in $S$, and let $\sigma$ be a simplex of $\calcp(S)$ of maximal dimension containing $b$ as its vertex. 
Since $\phi$ is injective, the equality $|\phi(\sigma)|=|\sigma|$ holds, and $\phi(\sigma)$ is thus a simplex of maximal dimension. 
By Proposition \ref{max_dim}, each vertex of $\phi(\sigma)$ is an HBP in $S$.
Assertion (i) is proved.

Let $b_1$ and $b_2$ be disjoint and equivalent HBPs in $S$. By Proposition \ref{max_dim}, we can find a simplex $\tau$ of $\calcp(S)$ of maximal dimension containing $b_1$ and $b_2$ as its vertices. Since the dimension of $\phi(\tau)$ is maximal in $\calcp(S)$, all vertices of $\phi(\tau)$ are equivalent by Proposition \ref{max_dim}. 
Assertion (ii) is proved.
\end{proof}

\begin{lem}\label{HBC_to_HBC}
Let $S=S_{g,p}$ be a surface with $g \geq 2$ and $p \geq 2$.
Then any automorphism of $\calcp(S)$ preserves HBCs in $S$.
\end{lem}

\begin{proof}
The lemma follows because for any automorphism $\phi$ of $\calcp(S)$, $\phi$ and $\phi^{-1}$ preserve HBPs in $S$ by Lemma \ref{HBP_to_HBP} (i).
\end{proof}

A verbatim proof shows the following:

\begin{lem}
Let $S=S_{g,p}$ be a surface with $g \geq 2$ and $p \geq 2$, and let $\psi \colon \calcp_s(S)\rightarrow \calcp_s(S)$ be a superinjective map.
Then the following assertions hold:
\begin{enumerate}
\item The map $\psi$ preserves HBPs in $S$.
\item If $b_1$ and $b_2$ are disjoint, equivalent and separating HBPs in $S$, then $\psi(b_1)$ and $\psi(b_2)$ are also equivalent.
\item If $\psi$ is an automorphism of $\calcp_s(S)$, then $\psi$ preserves HBCs in $S$.
\end{enumerate}
\end{lem}

To show that an automorphism of $\calcp(S)$ preserves more topological information, let us introduce the following terminology.

\begin{defn}
Let $S$ be a surface, and let $\sigma$ be a simplex of $\calcp(S)$ consisting of HBP-vertices.
We say that $\sigma$ is {\it rooted} if there exists a curve $\alpha$ in $S$ contained in any HBP of $\sigma$.
In this case, if $|\sigma|\geq 2$, then $\alpha$ is uniquely determined and called the {\it root curve} of $\sigma$.
\end{defn}

Rooted simplices were introduced in \cite{kida-tor} for the Torelli complex analogously.

\begin{lem}\label{rooted}
Let $S=S_{g,p}$ be a surface with $g \geq 2$ and $p \geq 2$. 
Then any superinjective map from $\calcp(S)$ into itself preserves rooted simplices. 
Moreover, the same conclusion holds for any superinjective map from $\calcp_s(S)$ into itself.
\end{lem}

\begin{proof}
Let $\phi \colon \calcp(S)\rightarrow \calcp(S)$ be a superinjective map.
We note that the maximal dimension of a rooted simplex of $\calcp(S)$ is equal to $p-1$.
Let $\sigma$ be a rooted simplex of $\calcp(S)$ consisting of HBP-vertices $a_1,a_2,\ldots,a_p$.
It suffices to prove that $\phi(\sigma)$ is rooted. 

For each $j=1,2,\ldots,p$, there exists an HBP $b_j$ in $S$ with
\[i(a_j, b_j)\ne 0\quad \textrm{and}\quad i(a_k, b_j)=0 \quad \textrm{for\ any}\ k\in \{ 1,\ldots, p\} \setminus \{ j\}.\]
Since $\phi$ is superinjective, for each $j=1, 2,\ldots, p$, we have
\[i(\phi(a_j),\phi(b_j))\ne 0\quad \textrm{and}\quad i(\phi(a_k),\phi(b_j))=0\quad \textrm{for\ any}\ k\in \{ 1,\ldots, p\} \setminus \{ j\}.\] 
For each $j=1,2,\ldots,p$, there thus exists a curve $c_j \in \phi(a_j)$ with $i(c_j, \phi(b_j))\ne 0$.

We note that for each $k\in \{ 1,\ldots, p\} \setminus \{ j\}$, the HBP $\phi(a_k)$ does not contain $c_j$. 
Let $c_0$ be the curve of $\phi(a_1)$ distinct from $c_1$.  
It then follows that $c_0,c_1,\ldots, c_p$ are mutually distinct, disjoint and HBP-equivalent in $S$. 
Proposition \ref{max_dim} implies that the simplex of $\calcp(S)$ consisting of all HBPs of two of $c_0,c_1,\ldots, c_p$ is of maximal dimension. 
The simplex $\phi(\sigma)$ thus consists of $p$ pairs of two of $c_0,c_1,\ldots, c_p$. 
For each $j=1,\ldots, p$, since $\phi(a_j)$ contains $c_j$ and does not contain $c_k$ for any $k \in \{ 1,\ldots, p\} \setminus \{ j\}$, we obtain the equality $\phi(a_j)=\{c_0,c_j\}$.

Along an argument of the same kind, we can prove the assertion for any superinjective map from $\calcp_s(S)$ into itself.
\end{proof}

\begin{lem}\label{root_curve}
Let $S=S_{g,p}$ be a surface with $g \geq 2$ and $p \geq 2$, and let $\phi \colon \calcp(S) \rightarrow \calcp(S)$ be a superinjective map.
Pick a simplex $\sigma$ of $\calcp(S)$ of maximal dimension, and let $\{ \alpha_0,\alpha_1,\ldots, \alpha_p\}$ denote the collection of curves in HBPs of $\sigma$. 
Then there exists a collection of curves in $S$, $\{ \beta_0,\beta_1,\ldots, \beta_p\}$, satisfying the equality
\[\phi(\{ \alpha_j, \alpha_k\})=\{ \beta_j, \beta_k\}\]
for any distinct $j, k=0,1,\ldots, p$.
Moreover, the same conclusion holds for any superinjective map from $\calcp_s(S)$ into itself and any simplex of $\calcp_s(S)$ of maximal dimension.
\end{lem}

\begin{proof}
For each $j=0, 1,\ldots, p$, we define $\sigma_j$ to be the rooted subsimplex of $\sigma$ of dimension $p-1$ that consists of all HBPs in $\sigma$ containing $\alpha_j$.
Let $\beta_j$ denote the root curve of the rooted simplex $\phi(\sigma_j)$.
Note that we have $\beta_j\neq \beta_k$ for any $j\neq k$ because the maximal dimension of a rooted simplex of $\calcp(S)$ is equal to $p-1$.
For any $j\neq k$, since $\phi(\{ \alpha_j, \alpha_k\})$ contains $\beta_j$ and $\beta_k$, we have the equality $\phi(\{ \alpha_j, \alpha_k\})=\{ \beta_j, \beta_k\}$.

Along an argument of the same kind, we can prove the assertion for any superinjective map from $\calcp_s(S)$ into itself.
\end{proof}

Applying observations so far on rooted simplices, we prove the following:

\begin{lem}\label{lem-phi-ns-s}
Let $S=S_{g,p}$ be a surface with $g \geq 2$ and $p \geq 2$. Then any automorphism of $\calcp(S)$ preserves non-separating HBPs in $S$ and separating HBPs in $S$, respectively.
\end{lem}

\begin{proof}
Let $\phi$ be an automorphism of $\calcp(S)$.
It suffices to prove that $\phi$ preserves non-separating HBPs in $S$.
Pick a simplex $\sigma$ of $\calcp(S)$ of maximal dimension consisting of non-separating HBPs in $S$.
Let $s=\{\alpha_0, \alpha_1,\ldots,\alpha_p\}$ be the set of non-separating curves in $S$ such that
\begin{itemize}
\item $\sigma$ consists of all pairs of two curves in $s$; and
\item for each $j=0,\ldots,p-1$, $\{\alpha_j,\alpha_{j+1}\}$ is a $1$-HBP in $S$. 
\end{itemize}
\begin{figure}
\begin{center}
\includegraphics[width=6cm]{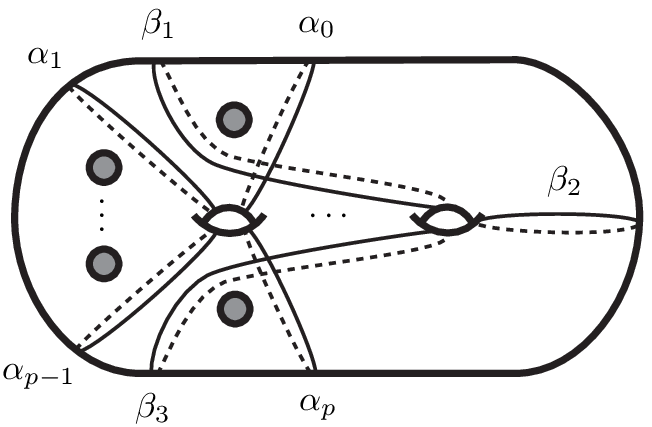}
\caption{}\label{fig-nshbp}
\end{center}
\end{figure}
As described in Figure \ref{fig-nshbp}, we can find non-separating curves $\beta_1$, $\beta_2$ and $\beta_3$ in $S$ such that
\begin{itemize}
\item $\beta_1$, $\beta_2$ and $\beta_3$ are mutually disjoint, distinct and HBP-equivalent in $S$;
\item $i(\beta_2, s)=i(\beta_1, s\setminus \{ \alpha_0\})=i(\beta_3, s\setminus \{ \alpha_p\})=0$;
\item $i(\beta_1, \alpha_0)\neq 0$ and $i(\beta_3, \alpha_p)\neq 0$;
\item $\beta_2$ is not HBP-equivalent to any curve of $s$; and
\item both $\{\beta_1, \beta_2\}$ and $\{ \beta_2, \beta_3\}$ are $1$-HBPs in $S$. 
\end{itemize}
By Lemma \ref{root_curve}, there exist curves $\gamma_j$, $\delta_k$ in $S$ with
\[\phi(\{ \alpha_{j_1}, \alpha_{j_2}\})=\{ \gamma_{j_1}, \gamma_{j_2}\},\quad \phi(\{ \beta_{k_1}, \beta_{k_2}\})=\{ \delta_{k_1}, \delta_{k_2}\}\]
for any distinct $j_1, j_2=0,\ldots, p$ and any distinct $k_1, k_2=1, 2, 3$.
We put $t=\{ \gamma_0,\ldots, \gamma_p\}$.
Since $\{ \delta_1, \delta_2\}$ intersects $\gamma_0$ and is disjoint from any curve in $t\setminus \{ \gamma_0\}$ and since $\{ \delta_2, \delta_3\}$ intersects $\gamma_p$ and is disjoint from any curve in $t\setminus \{ \gamma_p\}$, we see that $\delta_2$ is disjoint from any curve in $t$.
Since $\{ \delta_1, \delta_2\}$ is not equivalent to any HBP of two curves in $t\setminus \{ \gamma_0\}$, using Lemmas \ref{comp} and \ref{comp-s}, we see that $\{ \delta_1, \delta_2\}$ is a 1-HBP in $S$ and that there exists a $(p-1)$-HBP of two curves in $t\setminus \{ \gamma_0\}$ such that the other curves in it are contained in the holed sphere cut off by that $(p-1)$-HBP from $S$.
Similarly, $\{ \delta_2, \delta_3\}$ is a 1-HBP in $S$, and there exists a $(p-1)$-HBP of two curves in $t\setminus \{ \gamma_p\}$ such that the other curves in it are contained in the holed sphere cut off by that $(p-1)$-HBP from $S$.
It thus follows that $\{ \gamma_0, \gamma_p\}$ is a $p$-HBP in $S$.

We now assume that each $\gamma_j$ is a separating curve in $S$.
If $\delta_2$ lies in the component of $S_{\{ \gamma_0, \gamma_p\}}$ of positive genus that contains $\gamma_0$ as a boundary component, then the HBP $\{ \delta_2, \delta_3\}$ cannot intersect $\gamma_p$, kept disjoint from any curve in $t\setminus \{ \gamma_p\}$, by Lemma \ref{comp-s}. 
This is a contradiction.
In a similar way, we can deduce a contradiction if we assume that $\delta_2$ lies in the component of $S_{\{ \gamma_0, \gamma_p\}}$ of positive genus that contains $\gamma_p$ as a boundary component.
\end{proof}

\begin{lem}
Let $S=S_{g,p}$ be a surface with $g \geq 2$ and $p \geq 2$. Then for any integers $j$, $k$ with $2\leq j\leq p$ and $1\leq k\leq p$, any automorphism of $\calcp(S)$ preserves $j$-HBCs in $S$ and $k$-HBPs in $S$, respectively. 
Moreover, the same conclusion holds for any automorphism of $\calcp_s(S)$.
\end{lem}

\begin{proof}
Let $\phi$ be an automorphism of $\calcp(S)$.
Pick a $j$-HBC $\alpha$ in $S$ with $2\leq j\leq p$, and suppose that $\phi(\alpha)$ is a $j'$-HBC in $S$ with $2\leq j'\leq p$.
Let $R$ (resp.\ $R'$) be the component of $S_{\alpha}$ (resp.\ $S_{\phi(\alpha)}$) of positive genus, which is a surface of genus $g$ with $p-j+1$ (resp.\ $p-j'+1$) boundary components. We note that each HBP in $R$ can be identified with an HBP in $S$ via the inclusion of $R$ into $S$. The same thing holds for $R'$.
Let $\sigma$ be a simplex of $\calcp(R)$ of maximal dimension, which is identified with a simplex of $\calcp(S)$ consisting of HBP-vertices.
Since each HBP in $\phi(\sigma)$ is disjoint from the HBC $\phi(\alpha)$, we can identify $\phi(\sigma)$ with a simplex of $\calcp(R')$. 
We then obtain the inequality
\[\binom{p-j+2}{2}=|\sigma|=|\phi(\sigma)|\leq \binom{p-j'+2}{2},\]
which implies $j\geq j'$. The same argument for $\phi^{-1}$ shows $j\leq j'$.
We thus conclude the equality $j=j'$.

We prove that $\phi$ preserves $k$-HBPs in $S$ for each integer $k$ with $1\leq k\leq p$ by induction on $p$.
If $p=2$, then $\phi$ preserves 2-HBPs in $S$ because $\phi$ preserves 2-HBCs in $S$ and because any HBP in $S$ disjoint from a 2-HBC in $S$ is a 2-HBP in $S$.
It then follows that $\phi$ preserves 1-HBPs in $S$.

We next assume $p\geq 3$.
Pick a simplex $\sigma$ of maximal dimension in $\calcp(S)$ consisting of non-separating HBPs in $S$.
Let $s=\{ \alpha_0,\ldots, \alpha_p\}$ denote the collection of curves in HBPs of $\sigma$ so that $\{ \alpha_j, \alpha_{j+1}\}$ is a 1-HBP in $S$ for each $j=0,\ldots, p-1$.
By Lemma \ref{root_curve}, there exist curves $\beta_0,\ldots, \beta_p$ in $S$ with $\phi(\{ \alpha_j, \alpha_k\})=\{ \beta_j, \beta_k\}$ for any distinct $j, k=0, \ldots, p$.
Choose two distinct 2-HBCs $\gamma_1$, $\gamma_2$ in $S$ contained in the holed sphere cut off by $\{ \alpha_0, \alpha_2\}$ from $S$.
We now apply the hypothesis of the induction to the component of $S_{\gamma_1}$ of positive genus.
It then follows that each of $\{ \beta_0, \beta_2\}$ and $\{ \beta_j, \beta_{j+1}\}$ for any $j=2,\ldots, p-1$ is a 1-HBP in the component of $S_{\phi(\gamma_1)}$ of positive genus and that $\{ \beta_0, \beta_p\}$ is a $(p-1)$-HBP in that component.

Suppose that $\{ \beta_j, \beta_{j+1}\}$ is a 2-HBP in $S$ for some $j=2,\ldots, p-1$.
We choose a curve $\alpha$ in $S$ such that $\{ \alpha_0, \alpha \}$ is an HBP in $S$; and $\alpha$ intersects $\alpha_{j+1}$ and is disjoint from $\alpha_k$ for any $k\in \{ 0,\ldots, p\} \setminus \{ j+1\}$.
Note that $\{ \alpha_0, \alpha \}$ is disjoint from $\gamma_1$ and $\gamma_2$.
On the other hand, $\phi(\gamma_1)$ and $\phi(\gamma_2)$ fill the holed sphere cut off by the 2-HBP $\{ \beta_j, \beta_{j+1}\}$ in $S$.
Since the 2-simplex of $\calcp(S)$ consisting of $\{ \alpha_0, \alpha \}$, $\{ \alpha_0, \alpha_1\}$ and $\{ \alpha_0, \alpha_2\}$ is rooted, the HBP $\phi(\{ \alpha_0, \alpha\})$ contains $\beta_0$.
Another curve of $\phi(\{ \alpha_0, \alpha \})$ intersects $\beta_{j+1}$ and thus intersects $\phi(\gamma_1)$ or $\phi(\gamma_2)$.
This is a contradiction.

We thus proved that $\{ \beta_0, \beta_2\}$ is a 2-HBP in $S$ and that $\{ \beta_j, \beta_{j+1}\}$ is a 1-HBP in $S$ for each $j=2,\ldots, p-1$.
It follows that $\{ \beta_0, \beta_1\}$ and $\{ \beta_1, \beta_2\}$ are 1-HBPs in $S$.
For each $k=1,\ldots, p$, the map $\phi$ therefore preserves non-separating $k$-HBPs in $S$.

If non-separating HBPs are replaced by separating HBPs in the above argument, then we can prove that $\phi$ preserves separating $k$-HBPs in $S$ for each $k$.
A verbatim proof shows that any automorphism of $\calcp_s(S)$ preserves $j$-HBCs in $S$ for each $j$, and then shows that it also preserves separating $k$-HBPs in $S$ for each $k$.
\end{proof}


\section{Construction of $\Phi$ in the case $p=2$}\label{sec-const}

Let $S=S_{g, 2}$ be a surface with $g\geq 2$.
For an automorphism $\phi$ of $\calcp(S)$, we define a map $\Phi \colon V(S) \rightarrow V(S)$ as in Section \ref{subsec-plan}, which is shown to be well-defined throughout this section.
In Sections \ref{subsec-pent} and \ref{subsec-hex}, we study pentagons in $\calcp(S)$ and hexagons in $\calcp_s(S)$, respectively.
They are used to show that $\Phi$ is well-defined on the set of non-separating curves in $S$ and on the set of separating curves in $S$ which are not HBCs in $S$, respectively.
In Section \ref{subsec-defn}, we prove that $\Phi$ is well-defined.

\subsection{Pentagons in $\calcp(S)$}\label{subsec-pent}

We mean by a {\it pentagon} in $\calcp(S)$ the full subgraph of $\calcp(S)$ spanned by five vertices $v_1,\ldots, v_5$ with $i(v_k, v_{k+1})=0$ and $i(v_k, v_{k+2})\neq 0$ for each $k$ mod $5$ (see Figure \ref{fig_pen}).
\begin{figure}
\begin{center}
\includegraphics[height=5cm]{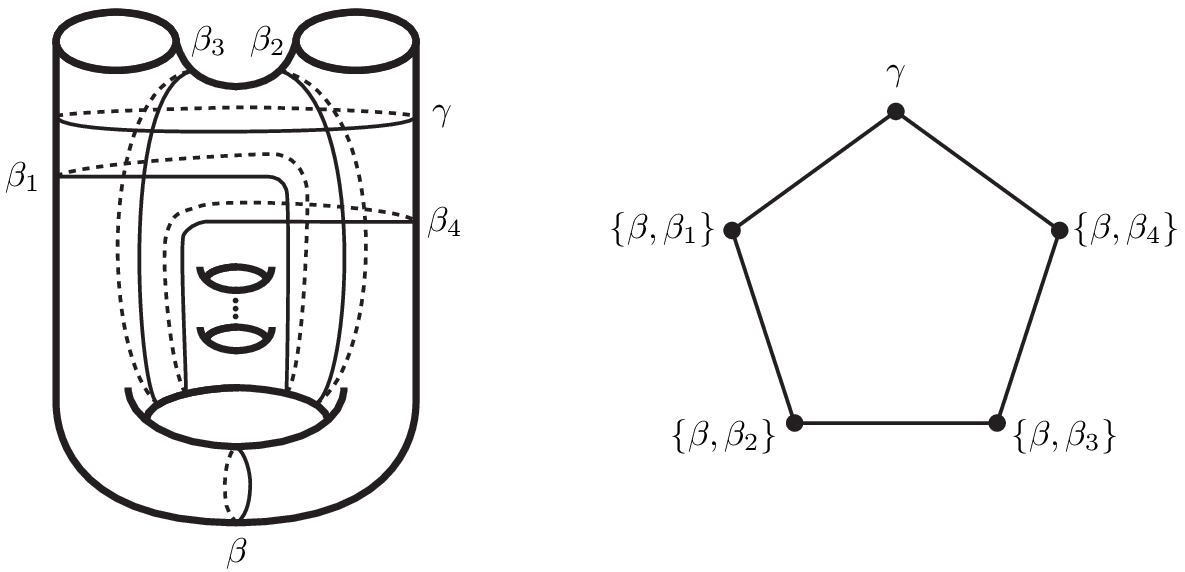}
\caption{A pentagon in $\calcp(S)$}\label{fig_pen}
\end{center}
\end{figure}
In this case, let us say that the pentagon is defined by the 5-tuple $(v_1,\ldots, v_5)$.

Let $\phi$ be an automorphism of $\calcp(S)$, and pick a non-separating curve $\alpha$ in $S$.
To define $\Phi(\alpha)\in V(S)$, we choose two disjoint and distinct HBPs $a_1=\{ \alpha, \alpha_1\}$ and $a_2=\{ \alpha, \alpha_2\}$ in $S$ and have to show that the root curve of the edge $\{ \phi(a_1), \phi(a_2)\}$ of $\calcp(S)$ depends only on $\alpha$.
We connect any two edges of $\calcp(S)$ consisting of two HBPs in $S$ containing $\alpha$ by a sequence of pentagons in $\calcp(S)$ such that
\begin{itemize}
\item each pentagon in that sequence is equal to the one described in Figure \ref{fig_pen} up to a homeomorphism of $S$; and
\item any two successive pentagons in that sequence share at least two HBPs. 
\end{itemize}
In Lemma \ref{lem-pen2}, considering the image of this sequence of pentagons via $\phi$, we obtain the aforementioned assertion.

\begin{lem}\label{lem-pen}
Let $S=S_{g, 2}$ be a surface with $g \geq 2$. 
For each $k=1,2,3$, let $a_k=\{\alpha, \alpha_k\}$ be a non-separating HBP in $S$ such that $\{ a_1, a_2\}$ and $\{ a_2, a_3\}$ are edges of $\calcp(S)$. 
Then there exists a sequence of pentagons in $\calcp(S)$, $\Pi_1,\Pi_2, \ldots, \Pi_n$, satisfying the following: For each $k=1,2,\ldots,n$,
\begin{enumerate}
\item up to a homeomorphism of $S$, $\Pi_k$ is equal to the pentagon in Figure \ref{fig_pen} that is defined by a $5$-tuple consisting of a $2$-HBC, a $2$-HBP, a $1$-HBP, a $1$-HBP and a $2$-HBP in this order; 
\item any of the four HBPs of $\Pi_k$ contains $\alpha$;
\item we have $a_1 \in \Pi_1, a_3 \in \Pi_n$ and $a_2 \in \Pi_k$; and
\item if $k<n$, then $\Pi_k$ and $\Pi_{k+1}$ share at least two HBPs.
\end{enumerate}
\end{lem}

\begin{proof}
We first deal with the case where $a_2$ is a $2$-HBP and subsequently the case where $a_2$ is a $1$-HBP. 

Let us assume that $a_2$ is a $2$-HBP. 
Since any distinct two $2$-HBPs in $S$ intersect, $a_1$ and $a_3$ are both $1$-HBPs. 
We can find a pentagon $\Pi$ in $\calcp(S)$ containing $a_1$ and $a_2$ as its vertices and equal to the one in Figure \ref{fig_pen} up to a homeomorphism of $S$. 
Let $b$ denote the $2$-HBC in $\Pi$, and let $Q$ denote the component of $S_{a_2}$ of genus zero, which is homeomorphic to $S_{0,4}$. 
The set $V(Q)$ then contains $\alpha_1$ and $\alpha_3$. 
Let $h \in \mod(S)$ be the half twist about $b$ exchanging the two components of $\partial S$.
Let $x \in \mod(S)$ be the Dehn twist about $\alpha_1$. 
Define $\Gamma$ to be the subgroup of $\mod(S)$ generated by $h$ and $x$. 
Since $a_2$ is fixed by $\Gamma$, there exists a natural homomorphism $p \colon \Gamma \rightarrow \mod(Q)$. 

We denote by $\mod(Q;\alpha,\alpha_2)$ the subgroup of $\mod(Q)$ consisting of all elements that fix each of the two components of $\partial Q$ corresponding to $\alpha$ and $\alpha_2$. 
We show the equality $p(\Gamma)=\mod(Q;\alpha,\alpha_2)$. 
The element $p(h)$ is a half twist about $b \in V(Q)$, and $p(x)$ is the Dehn twist about $\alpha_1 \in V(Q)$. 
We thus have the inclusion $p(\Gamma)<\mod(Q;\alpha,\alpha_2)$. 
Since the Dehn twists about $b$ and $\alpha_1$ generate $\pmod(Q)$ and since $p(h)$ exchanges the two components of $\partial S$, we obtain the inclusion $\mod(Q;\alpha,\alpha_2)<p(\Gamma)$. 
Our claim follows. 

\begin{figure}
\begin{center}
\includegraphics[width=7cm]{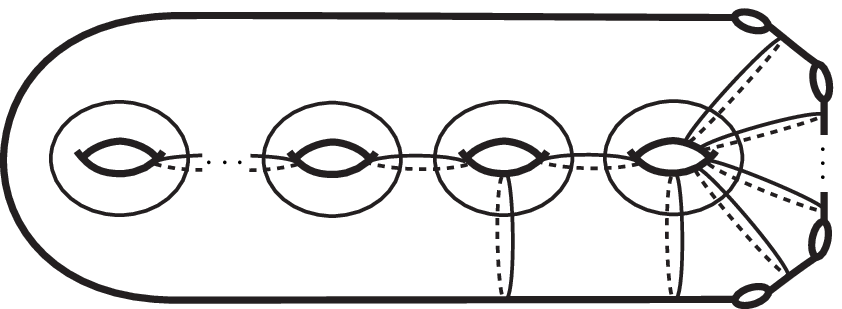}
\end{center}
\caption{$\pmod$ is generated by Dehn twists about these curves.}\label{fig-mcg}
\end{figure}

We put $H=\{ h^{\pm 1}, x^{\pm 1}\}$.
We note that $\alpha_1$ and $\alpha_3$ lie in $V(Q)$ and that $\alpha$ and $\alpha_2$ are contained in distinct components of $Q_{\alpha_1}$ (resp.\ $Q_{\alpha_3}$) as boundary components.
It follows that $\alpha_1$ and $\alpha_3$ lie in the same orbit for the action of $\mod(Q;\alpha,\alpha_2)$ on $V(Q)$. 
We can thus find $h_1,\ldots, h_n\in H$ with $\alpha_3=p(h_1)\cdots p(h_n)\alpha_1$.
Let us consider the following sequence of pentagons, 
\[\Pi,\; h_1\Pi,\; h_1h_2\Pi,\; \ldots,\; h_1h_2\cdots h_n\Pi.\]
This sequence satisfies conditions (i), (ii) and (iii) in the lemma. We now prove that it satisfies condition (iv). 
Let $c_1$ and $c_2$ denote the two vertices of $\Pi$ except $a_1, a_2$ and $b$ so that $c_j$ is a $j$-HBP for each $j=1,2$. 
It then follows that $\Pi$ is the $5$-tuple consisting of the five vertices $b,a_2,a_1,c_1,c_2$ in this order. 
Since $a_2$ and $c_2$ are disjoint from $b$, the three vertices $a_2$, $c_2$ and $b$ are fixed by $h$. 
Since $a_2$ and $c_1$ are disjoint from $a_1$, the three vertices $a_2$, $a_1$ and $c_1$ are fixed by $x$. 
Our sequence of pentagons therefore satisfies condition (iv).

We next suppose that $a_2$ is a $1$-HBP. 
Let $R$ denote the component of $S_{a_2}$ of positive genus, which is homeomorphic to $S_{g-1,3}$.
Note that $\alpha_1$ and $\alpha_3$ are curves in $R$. 
We prove the existence of a sequence of pentagons in the lemma in the following three cases: (a) $a_1$ and $a_3$ are both $1$-HBPs; (b) $a_1$ and $a_3$ are both $2$-HBPs; and (c) one of $a_1$ and $a_3$ is a $1$-HBP and another is a $2$-HBP.

(a) Suppose that $a_1$ and $a_3$ are both $1$-HBPs. 
Pick a pentagon $\Pi$ containing $a_1$ and $a_2$ as its vertices and equal to the one in Figure \ref{fig_pen} up to a homeomorphism of $S$.
Let $b$, $c$ and $d$ denote the $2$-HBC of $\Pi$, the $2$-HBP of $\Pi$ disjoint from $a_2$ and the other $2$-HBP of $\Pi$, respectively. 
It then follows that $\Pi$ is defined by the $5$-tuple consisting of the five vertices $b$, $c$, $a_2$, $a_1$ and $d$ in this order. We put $c=\{\alpha,\gamma\}$.

Recall that in general, if $X$ is a surface of positive genus described in Figure \ref{fig-mcg}, then the set of Dehn twists about curves in Figure \ref{fig-mcg} generates $\pmod(X)$ (see \cite{gervais}).  
One can then find a set of curves in $R$, denoted by $U$, satisfying the following two conditions (see Figure \ref{fig-penta}): 
\begin{itemize}
\item The set of Dehn twists about curves in $U$ generates $\pmod(R)$;
\item There exists a curve $\delta_1$ in $U$ with $i(\alpha_1, \delta_1)\neq 0$ and $i(\alpha_1, \delta)=0$ for any $\delta \in U\setminus \{ \delta_1\}$; and
\item There exists a curve $\delta_2$ in $U$ with $i(\gamma, \delta_2)\neq 0$ and $i(\gamma, \delta)=0$ for any $\delta \in U\setminus \{ \delta_2\}$.
\end{itemize}
Let $T$ denote the subset of $\pmod(S)$ consisting of all Dehn twists about curves in $U$ and their inverses, where curves in $R$ are naturally identified with curves in $S$.
Let $\Lambda$ be the subgroup of $\pmod(S)$ generated by $T$, and define $q\colon \Lambda \rightarrow \pmod(R)$ as the natural homomorphism. 
Since $\alpha_1$ and $\alpha_3$ are both $2$-HBCs in $R$ cutting off a pair of pants containing $\alpha$ and a component of $\partial S$ as boundary components, they lie in the same orbit for the action of $\pmod(R)$ on $V(R)$. 
Choose $g_1,\ldots, g_m\in T$ with $\alpha_3=q(g_1)\cdots q(g_m)\alpha_1$. 
The sequence of pentagons,
\[\Pi,  \; g_1\Pi, \; g_1g_2\Pi,\; \ldots, \; g_1g_2\cdots g_m\Pi,\]
then satisfies conditions (i), (ii) and (iii) in the lemma. 
Note that $a_1$ is fixed by each element of $T$ except $t_{\delta_1}$ and its inverse and that $c$ is fixed by each element of $T$ except $t_{\delta_2}$ and its inverse. 
Since $a_2$ is fixed by $\Lambda$, the above sequence of pentagons satisfies condition (iv).

(b) Suppose that $a_1$ and $a_3$ are both $2$-HBPs. 
Pick a pentagon $\Pi$ containing $a_1$ and $a_2$ as its vertices and equal to the one in Figure \ref{fig_pen} up to a homeomorphism of $S$.
Let $c$ denote the $1$-HBP of $\Pi$ disjoint and distinct from $a_2$, and put $c=\{ \alpha, \gamma \}$. 
Along an argument similar to that in case (a), we obtain a desired sequence of pentagons (see also Figure \ref{fig-penta}). 

\begin{figure}
\begin{center}
\includegraphics[width=10cm]{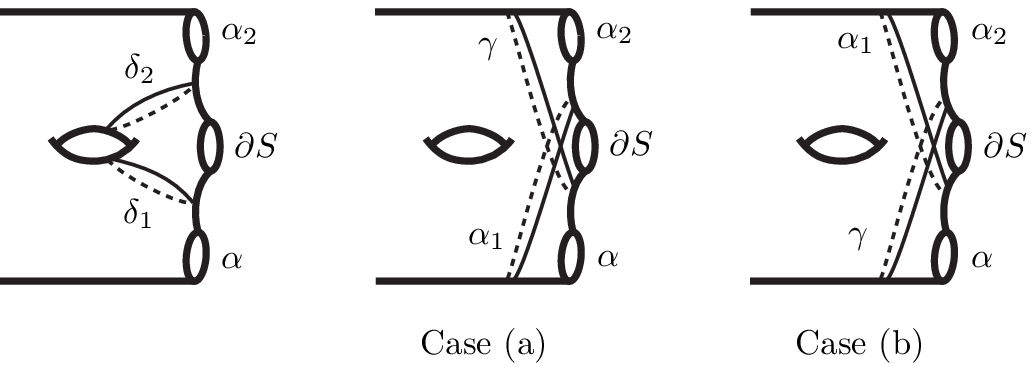}
\end{center}
\caption{}\label{fig-penta}
\end{figure}

(c) Finally, we assume that one of $a_1$ and $a_3$ is a $1$-HBP and another is a $2$-HBP.
We may assume that $a_1$ is a $1$-HBP and $a_3$ is a $2$-HBP. 
Choose a pentagon $\Pi$ containing $a_1$ and $a_2$ as its vertices and equal to the one in Figure \ref{fig_pen} up to a homeomorphism of $S$. 
We denote by $b$, $c$ and $d$ the $2$-HBC of $\Pi$, the $2$-HBP of $\Pi$ disjoint from $a_2$ and the other $2$-HBP of $\Pi$, respectively. 
The pentagon $\Pi$ is then defined by the $5$-tuple consisting of $b$, $c$, $a_2$, $a_1$ and $d$ in this order.
We put $c=\{\alpha,\gamma\}$. Note that $\gamma$ and $\alpha_3$ are in the same orbit for the action of $\pmod(R)$ on $V(R)$.
Along an argument of the same kind as in case (a) where $\alpha_1$ and $\alpha_3$ are in the same orbit for the action of $\pmod(R)$ on $V(R)$, we obtain the lemma.
\end{proof}

Let us introduce terminology used in the subsequent proposition.
Let $X=S_{g, p}$ be a surface with $p\geq 2$.
\begin{itemize}
\item Given two distinct components $\partial_1$ and $\partial_2$ of $\partial X$ and a separating curve $\alpha$ in $X$, we say that $\alpha$ {\it separates} $\partial_1$ and $\partial_2$ if $\partial_1$ and $\partial_2$ are contained in distinct components of $X_{\alpha}$.
\item Let $\partial_1,\ldots, \partial_k$ be pairwise distinct components of $\partial X$.
We say that an HBC $\beta$ in $X$ {\it encircles} $\partial_1,\ldots, \partial_k$ if $\beta$ cuts off from $S$ a holed sphere containing $\partial_1,\ldots, \partial_k$ and homeomorphic to $S_{0, k+1}$.
\end{itemize}

\begin{prop}\label{connected}
Let $X=S_{g,p}$ be a surface with $g\geq 1$ and $p\geq 4$, and let $\partial_1$ and $\partial_2$ be two distinct components of $\partial X$.
Then the full subcomplex $\cal{D}$ of $\calc(X)$ spanned by all vertices that correspond to HBCs in $X$ separating $\partial_1$ and $\partial_2$ is connected.
\end{prop}

\begin{proof}
The proof is based on Lemma 2.1 of \cite{putman-conn}, which presents a technique to prove connectivity of a simplicial complex on which $\pmod(X)$ acts.
Pick two distinct components $\partial_3$ and $\partial_4$ of $\partial X$ other than $\partial_1$ and $\partial_2$.

\begin{claim}\label{claim-d}
For any HBC $\alpha$ in $X$ separating $\partial_1$ and $\partial_2$, there exists a path in $\cal{D}$ connecting $\alpha$ to a vertex corresponding to an HBC in $X$ encircling $\partial_2$ and $\partial_3$. 
\end{claim}

\begin{proof}
If the component of $X_{\alpha}$ containing $\partial_2$ contains $\partial_3$, then one can find an HBC in $X$ disjoint from $\alpha$ and encircling $\partial_2$ and $\partial_3$.
Otherwise, one can find a path in $\cal{D}$, $\alpha$, $\alpha_1$, $\alpha_2$, $\alpha_3$, $\alpha_4$, such that $\alpha_1$ encircles $\partial_1$ and $\partial_3$; $\alpha_2$ encircles $\partial_1$, $\partial_3$ and $\partial_4$; $\alpha_3$ encircles $\partial_1$ and $\partial_4$; and $\alpha_4$ encircles $\partial_2$ and $\partial_3$.
\end{proof}

\begin{claim}\label{claim-23-conn}
For any two HBCs $\alpha$ and $\beta$ in $X$ encircling $\partial_2$ and $\partial_3$, there exists a path in $\cal{D}$ connecting $\alpha$ and $\beta$.
\end{claim}

\begin{proof}
Let $U$ be the set of curves in $X$ described in Figure \ref{fig-mcg}.
As already mentioned in the proof of Lemma \ref{lem-pen}, $\pmod(X)$ is generated by Dehn twists about curves in $U$.
After labeling $\partial_1,\ldots, \partial_p$ appropriately in Figure \ref{fig-mcg}, one can find an HBC $\gamma$ in $X$ encircling $\partial_2$ and $\partial_3$ such that
\begin{itemize}
\item there exists a unique curve $\delta$ in $U$ intersecting $\gamma$; and
\item there exists an HBC $\gamma'$ in $X$ which separates $\partial_1$ and $\partial_2$ and is disjoint from $\gamma$ and $\delta$.
\end{itemize}
The vertices $t_{\delta}\gamma$, $\gamma'$ and $\gamma$ of $\cal{D}$ form a path in $\cal{D}$.
Note that $\pmod(X)$ sends $\gamma$ to any HBC in $X$ encircling $\partial_2$ and $\partial_3$.
The claim is now obtained in the same way as in the construction of a sequence of pentagons in the proof of Lemma \ref{lem-pen}.
\end{proof}

The proposition follows from Claims \ref{claim-d} and \ref{claim-23-conn}.
\end{proof}

\begin{lem}\label{non-sep_conn}
Let $Y=S_{g, p}$ be a surface with $g \geq 2$ and $p \geq 2$, and pick a non-separating curve $\alpha$ in $Y$. 
Then the full subcomplex of $\calcp(Y)$ spanned by all vertices that correspond to HBPs in $Y$ containing $\alpha$ is connected.
\end{lem}

\begin{proof}
We set $X=Y_{\alpha}$, which is homeomorphic to $S_{g-1, p+2}$.
Using the natural one-to-one correspondence between HBPs in $Y$ containing $\alpha$ and HBCs in $X$ which separates the two boundary components of $X$ corresponding to $\alpha$, one can deduce the lemma from Proposition \ref{connected}. 
\end{proof}

The following lemma gives information on the image of the pentagon in Figure \ref{fig_pen} via an automorphism $\phi$ of $\calcp(S)$.

\begin{lem}\label{root_penta}
Let $S=S_{g, 2}$ be a surface with $g\geq 2$. Let $(a, b_2, b_1, c_1, c_2)$ be a $5$-tuple defining a pentagon $\Pi$ in $\calcp(S)$ such that
\begin{itemize}
\item $b_j$ and $c_j$ are non-separating $j$-HBPs in $S$ for each $j=1, 2$; and
\item each of the three edges $\{ b_2, b_1\}$, $\{ b_1, c_1\}$ and $\{ c_1, c_2\}$ is rooted. 
\end{itemize}
Then the root curves of the three edges in the second condition are equal.
\end{lem}

\begin{proof}
Let $\alpha$, $\beta$ and $\gamma$ be the root curves of the edges $\{ b_2, b_1\}$, $\{ b_1, c_1\}$ and $\{ c_1, c_2\}$, respectively. 
Because of $i(b_1,c_1)=0$, the equality $i(\alpha,\beta)=i(\beta, \gamma)=i(\gamma, \alpha)=0$ holds. 
We deduce a contradiction by assuming that either (a) $\alpha$, $\beta$ and $\gamma$ are mutually distinct; (b) $\alpha=\beta \ne \gamma$ or $\beta=\gamma \ne \alpha$ holds; or (c) $\alpha=\gamma \ne \beta$ holds.

(a) We assume that $\alpha$, $\beta$ and $\gamma$ are mutually distinct. 
The equalities $b_1=\{\alpha,\beta\}$ and $c_1=\{\beta,\gamma\}$ then holds.
The HBP $\{\alpha,\gamma \}$ is a $2$-HBP because $b_1$ and $c_1$ are both $1$-HBPs. 
Since $b_1$ intersects $a$ and since $b_2$ is disjoint from $a$ and contains $\alpha$, we see that $\beta$ intersects $a$. 
Put $b_2=\{\alpha,\alpha_2\}$ and denote by $R$ the component of $S_{b_2}$ of genus zero, which is homeomorphic to $S_{0, 4}$. 
Note that $a$ (or a curve in $a$ if $a$ is an HBP in $S$) and $\beta$ lie in $R$ and fill $R$.
Since $\gamma$ is disjoint from $a$ and $\beta$, the equality $\alpha_2=\gamma$ holds. 
This is a contradiction because the equality implies that $b_2$ and $c_1$ are disjoint. 

(b) We suppose $\alpha=\beta \ne \gamma$. 
The equality $c_1=\{\beta, \gamma \}$ then holds. 
It follows from $i(b_2, a)=i(c_2, a)=0$ that we have $i(\alpha, a)=i(\gamma, a)=0$ and thus $i(c_1, a)=0$. This is a contradiction. 
We can also deduce a contradiction in the same manner if $\beta=\gamma \ne \alpha$ holds. 

\begin{figure}
\begin{center}
\includegraphics[height=5cm]{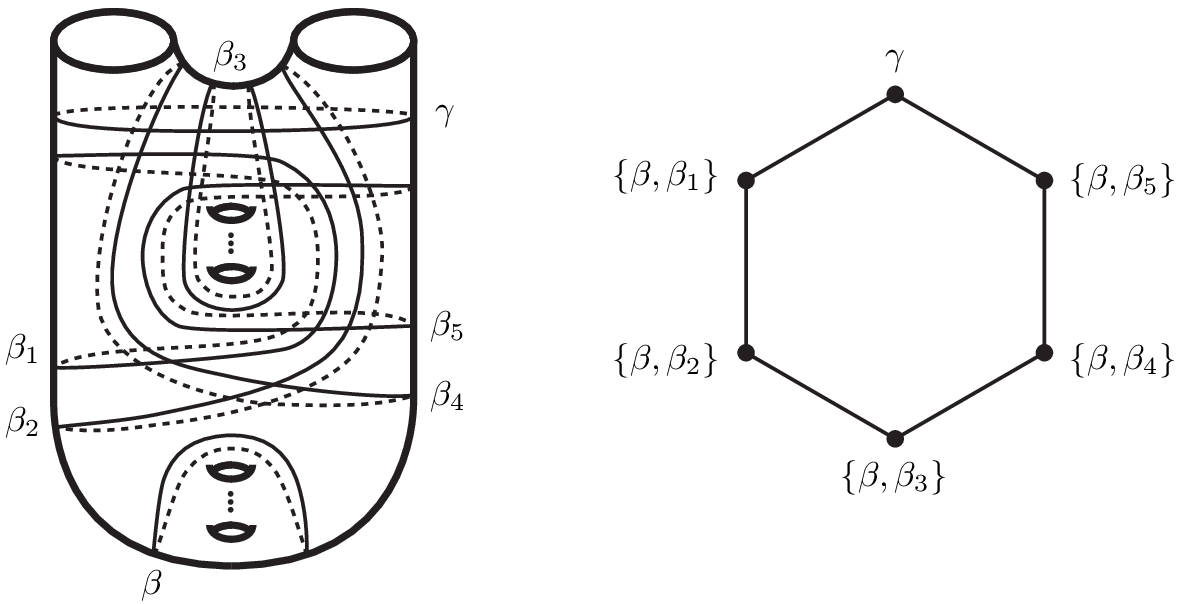}
\end{center}
\caption{A hexagon in $\calcp_s(S)$}\label{fig_hex}
\end{figure}

(c) If $\alpha=\gamma \ne \beta$, then we have $b_1=\{\alpha, \beta \}=c_1$.
This is a contradiction.
\end{proof}

\begin{lem}\label{lem-pen2}
Let $S=S_{g, 2}$ be a surface with $g \geq 2$, and let $\phi$ be an automorphism of $\calcp(S)$. 
For each $k=1,2,3,4$, let $a_k=\{\alpha, \alpha_k\}$ be a non-separating HBP in $S$ such that $\{ a_1, a_2\}$ and $\{ a_3, a_4\}$ are edges of $\calcp(S)$. 
Then the root curves of the two edges $\{ \phi(a_1), \phi(a_2)\}$ and $\{ \phi(a_3), \phi(a_4)\}$ of $\calcp(S)$ are equal. 
\end{lem}

\begin{proof}
By Lemma \ref{non-sep_conn}, there exists a sequence of non-separating HBPs in $S$, $a_2=b_2, b_3,\ldots,b_{m-1}=a_3$, such that any two successive HBPs are disjoint and each $b_k$ contains $\alpha$.
We set $b_1=a_1$ and $b_m=a_4$.  
For each $k=1,\ldots, m-2$, if we have $i(b_k,b_{k+2})\ne 0$, then there exists a sequence of pentagons in $\calcp(S)$ connecting the two edges $\{ b_k, b_{k+1}\}$, $\{ b_{k+1}, b_{k+2}\}$ and satisfying conditions in Lemma \ref{lem-pen}.
Otherwise at least two of $b_k$, $b_{k+1}$ and $b_{k+2}$ are equal. 
We then obtain a sequence of pentagons in $\calcp(S)$, $\Pi_1, \Pi_2,\ldots, \Pi_n$, with $a_1, a_2\in \Pi_1$ and $a_3, a_4\in \Pi_n$.
By Lemma \ref{root_penta}, $\phi(\Pi_k)$ is a pentagon consisting of one HBC and four HBPs sharing a single non-separating curve in $S$. 
Since $\phi(\Pi_k)$ and $\phi(\Pi_{k+1})$ share at least two HBPs for each $k$, the non-separating curve shared by all HBPs of $\phi(\Pi_k)$ is equal to that of $\phi(\Pi_{k+1})$. 
The lemma follows. 
\end{proof}


\subsection{Hexagons and squares in $\calcp_s(S)$}\label{subsec-hex}

We mean by a {\it hexagon} in $\calcp_s(S)$ the full subgraph of $\calcp_s(S)$ spanned by six vertices $v_1,\ldots, v_6$ with $i(v_k,v_{k+1})=0$, $i(v_k,v_{k+2})\ne 0$ and $i(v_k,v_{k+3})\ne 0$ for each $k$ mod $6$ (see Figure \ref{fig_hex}).
In this case, let us say that the hexagon is defined by the $6$-tuple $(v_1,\ldots,v_6)$.

Similarly, we mean by a {\it square} in $\calcp_s(S)$ the full subgraph of $\calcp_s(S)$ spanned by four vertices $v_1,\ldots,v_4$ with $i(v_k,v_{k+1})=0$ and $i(v_k,v_{k+2})\neq 0$ for each $k$ mod $4$ (see Figure \ref{fig_squ}).
\begin{figure}
\begin{center}
\includegraphics[height=5cm]{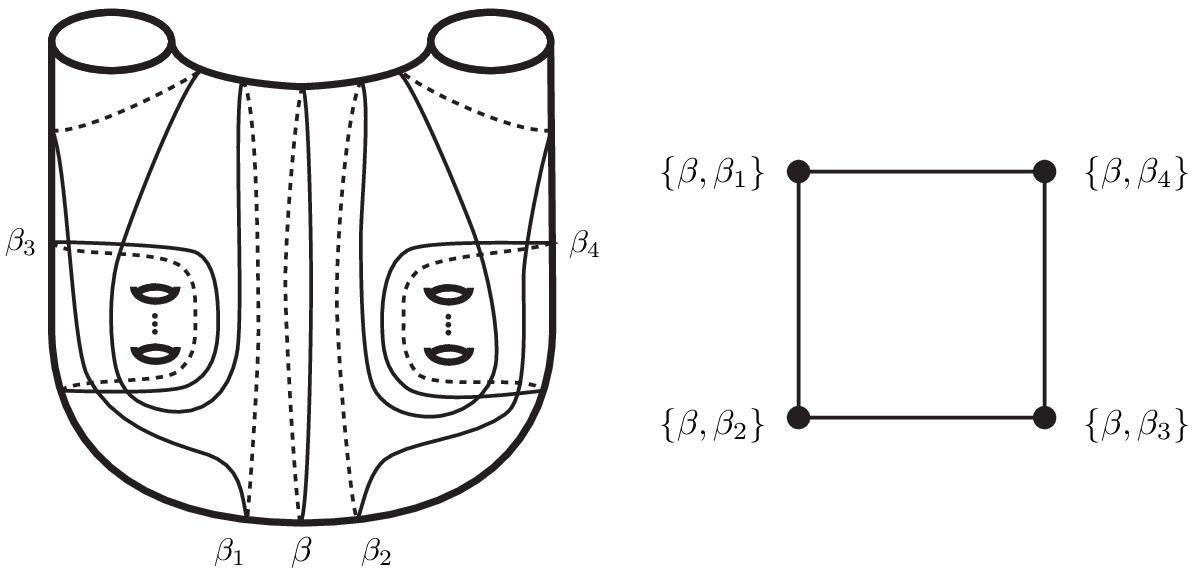}
\caption{A square in $\calcp_s(S)$}\label{fig_squ}
\end{center}
\end{figure}
In this case, we say that the square is defined by the $4$-tuple $(v_1,\ldots, v_4)$.

Hexagons and squares in $\calcp_s(S)$ are used to show that $\Phi(\alpha)$ is well-defined for each separating curve $\alpha$ in $S$ which is not an HBC in $S$.
A conclusion similar to Lemma \ref{lem-pen2} for such an $\alpha$ is established in Lemma \ref{sep_HBP_rooted}.
To prove it, we connect any two edges of $\calcp_s(S)$ consisting of two HBPs in $S$ containing $\alpha$ by a sequence of hexagons and squares in $\calcp_s(S)$.

\begin{lem}\label{a_2_2-HBP}
Let $S=S_{g,2}$ be a surface with $g\geq 2$. 
For each $k=1,2,3$, let $a_k=\{\alpha,\alpha_k\}$ be a separating HBP in $S$ such that $\{ a_1, a_2\}$ and $\{ a_2, a_3\}$ are edges of $\calcp_s(S)$. 
We assume that $\partial S$ is contained in a single component of $S_{\alpha}$. 
Then there exists a sequence of hexagons in $\calcp_s(S)$, $\Pi_1, \Pi_2, \ldots, \Pi_n$, satisfying the following: For each $k=1,2, \ldots, n$,
\begin{enumerate}
\item $\Pi_k$ is defined by a $6$-tuple consisting of a $2$-HBC, a $2$-HBP, a $1$-HBP, a $2$-HBP, a $1$-HBP and a $2$-HBP in this order;
\item any of the five HBPs of $\Pi_k$ contains $\alpha$;
\item $a_1 \in \Pi_1$, $a_3 \in \Pi_n$ and $a_2 \in \Pi_k$; 
\item if $k<n$, then $\Pi_k$ and $\Pi_{k+1}$ share at least two HBPs. 
\end{enumerate}
\end{lem}

\begin{proof}
We find a desired sequence of hexagons in the following two cases: (a) $a_2$ is a $2$-HBP; and (b) $a_2$ is a $1$-HBP.

(a) If $a_2$ is a $2$-HBP, then $a_1$ and $a_3$ are $1$-HBPs since any two distinct separating $2$-HBPs in $S$ intersect. 
We denote by $Q$ the component of $S_{a_2}$ homeomorphic to $S_{0,4}$. 
Since we have $i(\alpha_1,\alpha_2)=i(\alpha_3,\alpha_2)=0$, $\alpha_1$ and $\alpha_3$ are elements of $V(Q)$. 
As in Figure \ref{fig_hex}, we can find a hexagon $\Pi$ defined by a $6$-tuple $(v_1,\ldots, v_6)$ with $v_1$ a $2$-HBC; $v_2=a_2$; and $v_3=a_1$.
We note that $v_4$ is a $2$-HBP; $v_5$ is a $1$-HBP; and $v_6$ is a $2$-HBP.
The vertex $v_1$ lies in $V(Q)$ since it is disjoint from $v_2=a_2$. 
Let $h \in \mod(S)$ be the half twist about $v_1$ exchanging the two components of $\partial S$.
Let $x \in \mod(S)$ be the Dehn twist about $\alpha_1$. 
Define $\Gamma$ to be the subgroup of $\mod(S)$ generated by $h$ and $x$. 
Since $a_2$ is fixed by $\Gamma$, there exists a natural homomorphism $p \colon \Gamma \rightarrow \mod(Q)$. 
We denote by $\mod(Q;\alpha,\alpha_2)$ the subgroup of $\mod(Q)$ consisting of all elements that fix each of the two components of $\partial Q$ corresponding to $\alpha$ and $\alpha_2$. 
As in the proof of Lemma \ref{lem-pen}, we obtain the equality $p(\Gamma)=\mod(Q;\alpha,\alpha_2)$. 
Since $\alpha$ and $\alpha_2$ are contained in distinct components of $Q_{\alpha_1}$ (resp.\ $Q_{\alpha_3}$), $\alpha_3$ lies in the orbit of $\alpha_1$ for the action of $\mod(Q;\alpha,\alpha_2)$ on $V(Q)$. 
Setting $H=\{ h^{\pm 1}, x^{\pm 1}\}$, we can thus find $h_1,\ldots, h_n\in H$ with $\alpha_3=p(h_1)\cdots p(h_n)\alpha_1$. 
Along an argument of the same kind as in Lemma \ref{lem-pen}, we obtain a desired sequence of hexagons.

(b) We next suppose that $a_2$ is a $1$-HBP. 
Since $\partial S$ is contained in a single component of $S_{\alpha}$, both $a_1$ and $a_3$ are $2$-HBPs. 
Let $R$ denote the component of $S_{a_2}$ of positive genus and containing a component of $\partial S$. 
Note that the number of boundary components of $R$ is equal to two and that $\alpha_1$ and $\alpha_3$ are elements of $V(R)$. 
As in Figure \ref{fig_hex}, we can find a hexagon $\Pi$ defined by a $6$-tuple $(v_1,\ldots, v_6)$ with $v_1=a_1$; $v_2=a_2$; and $v_6$ a $2$-HBC.
Note that $v_3$ is a $2$-HBP; $v_4$ is a $1$-HBP; and $v_5$ is a $2$-HBP.
We put $v_3=\{\alpha,\beta_3\}$. 
It then follows that $\beta_3$ lies in $V(R)$. 

\begin{figure}
\begin{center}
\includegraphics[width=9cm]{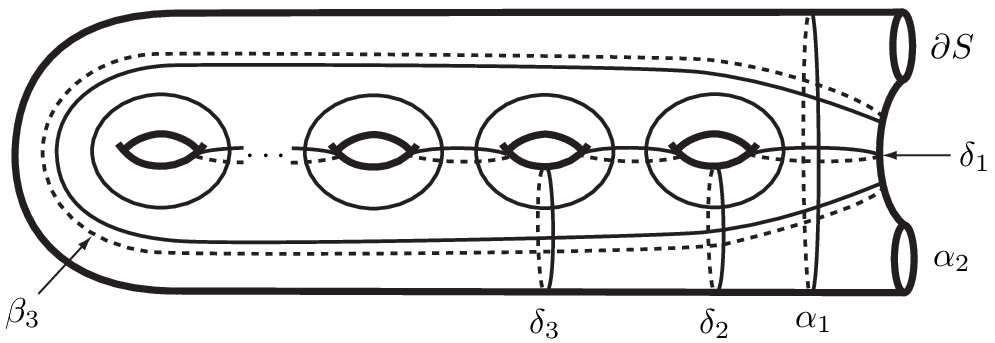}
\end{center}
\caption{}\label{fig-mcgr}
\end{figure}

Let $U$ be the set of the curves in $R$ described in Figure \ref{fig-mcgr} other than $\alpha_1$ and $\beta_3$, which satisfies the following two conditions: 
\begin{itemize}
\item The set of Dehn twists about curves in $U$ generates $\pmod(R)$; and
\item The curves $\delta_1$, $\delta_2$ and $\delta_3$ in $U$ satisfy $i(\alpha_1, \delta_1)\neq 0$ and $i(\alpha_1, \delta)=0$ for any $\delta \in U\setminus \{ \delta_1\}$; and $i(\beta_3, \delta_2)\neq 0$, $i(\beta_3, \delta_3)\neq 0$ and $i(\beta_3, \delta)=0$ for any $\delta \in U\setminus \{ \delta_2, \delta_3\}$.
\end{itemize}
Note that $\alpha_1$ and $\alpha_3$ lie in the same orbit for the action of $\pmod(R)$ on $V(R)$. Along an argument of the same kind as in Lemma \ref{lem-pen}, we obtain a desired sequence of hexagons.   
\end{proof}

\begin{lem}\label{a_2_1-HBP}
Let $S=S_{g,2}$ be a surface with $g\geq 2$. For each $k=1,2,3$, let $a_k=\{\alpha,\alpha_k\}$ be a separating HBP in $S$ such that $\{ a_1, a_2\}$ and $\{ a_2, a_3\}$ are distinct edges of $\calcp_s(S)$. 
We assume that each component of $S_{\alpha}$ contains a component of $\partial S$. 
Then there exists a square $\Pi$ in $\calcp_s(S)$ such that
\begin{itemize}
\item $a_1$, $a_2$ and $a_3$ are vertices of $\Pi$; and
\item the other vertex of $\Pi$ is a $1$-HBP in $S$ containing $\alpha$. 
\end{itemize}
\end{lem}

\begin{proof}
We first note that for each $k=1,2,3$, $a_k$ is a $1$-HBP in $S$ because each component of $S_{\alpha}$ contains a component of $\partial S$. 
It follows that $\alpha_1$ and $\alpha_3$ lie in the same component of $S_{\alpha}$ and that $\alpha_2$ lies in another component of $S_{\alpha}$, denoted by $R$. 
Choose a curve $\alpha_4$ in $R$ with $i(\alpha_4, \alpha_2)\neq 0$ and $a_4=\{\alpha,\alpha_4\}$ a $1$-HBP in $S$. 
The $4$-tuple $(a_1, a_2, a_3, a_4)$ then defines a square in $\calcp_s(S)$. 
\end{proof}

The following lemma is a variant of Lemma \ref{non-sep_conn} for separating HBPs.

\begin{lem}\label{conn_sep_2}
Let $X=S_{g,p}$ be a surface with $g\geq 2$ and $p\geq 2$, and pick a separating curve $\alpha$ in $X$ which is not an HBC in $X$.
Then the full subcomplex $\cal{E}$ of $\calcp_s(X)$ spanned by all vertices that correspond to separating HBPs in $X$ containing $\alpha$ is connected. 
\end{lem}

To prove this lemma, we need the following:

\begin{prop}\label{conn_1_3}
Let $Y=S_{g, p}$ be a surface with $g \geq 1$ and $p\geq 3$, and choose a component $\partial$ of $\partial Y$.
Then the full subcomplex $\cal{F}$ of $\calc(Y)$ spanned by all vertices that correspond to HBCs in $Y$ cutting off a holed sphere containing $\partial$ from $Y$ is connected.
\end{prop}

\begin{proof}
The proof is based on Lemma 2.1 in \cite{putman-conn} as in the proof of Proposition \ref{connected}.
Label components of $\partial Y$ as $\partial_1,\ldots, \partial_p$ with $\partial =\partial_1$.
We first claim that for any curve $\beta$ in $Y$ corresponding to a vertex of $\cal{F}$, there exists a path in $\cal{F}$ connecting $\beta$ to an HBC in $Y$ encircling $\partial_1$ and $\partial_2$.
If the holed sphere cut off by $\beta$ from $Y$ contains $\partial_2$, then one can find an HBC in $Y$ disjoint from $\beta$ and encircling $\partial_1$ and $\partial_2$.
Otherwise, choose $j\in \{ 3,\ldots, p\}$ so that $\partial_j$ is contained in the holed sphere cut off by $\beta$ from $Y$.
One can then find a path in $\cal{F}$, $\beta$, $\beta_1$, $\beta_2$, $\beta_3$, such that $\beta_1$ encircles $\partial_1$ and $\partial_j$; $\beta_2$ encircles $\partial_1$, $\partial_2$ and $\partial_j$; and $\beta_3$ encircles $\partial_1$ and $\partial_2$.
The claim follows.

To prove the proposition, it suffices to show that if $\gamma$ and $\delta$ are the curves described in Figure \ref{fig-econn}, then $t_{\delta}\gamma$ and $\gamma$ are connected by a path in $\cal{F}$.
\begin{figure}
\begin{center}
\includegraphics[width=10cm]{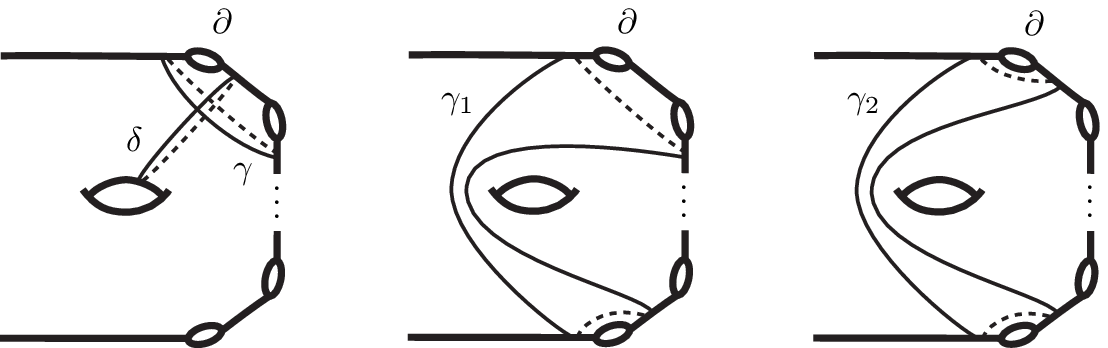}
\caption{}\label{fig-econn}
\end{center}
\end{figure}
Choose the two HBCs $\gamma_1$, $\gamma_2$ in $Y$ described in Figure \ref{fig-econn}.
The sequence $\gamma$, $\gamma_1$, $\gamma_2=t_{\delta}\gamma_2$, $t_{\delta}\gamma_1$, $t_{\delta}\gamma$ is then a desired path in $\cal{F}$.
\end{proof}

\begin{proof}[Proof of Lemma \ref{conn_sep_2}]
Let $a_1=\{ \alpha, \alpha_1\}$ and $a_2=\{ \alpha, \alpha_2\}$ be vertices of $\cal{E}$. To find a path in $\cal{E}$ connecting $a_1$ and $a_2$, we may assume $i(a_1, a_2) \ne 0$. 
Note that $\alpha_1$ and $\alpha_2$ lie in the same component of $X_\alpha$. 
Let $R$ denote that component of $X_\alpha$ and $R'$ denote another component of $X_\alpha$. 

If $R'$ contains a component of $\partial X$, then there exists a curve $\beta$ in $R'$ with $\{\alpha, \beta\}$ an HBP in $X$.
Since $\{\alpha, \beta\}$ is then disjoint from $a_1$ and $a_2$, we obtain the path $a_1$, $\{ \alpha, \beta \}$, $a_2$ in $\cal{E}$.

Suppose that $R'$ contains no component of $\partial X$. 
It then follows that $R$ contains at least two components of $\partial X$. 
By Proposition \ref{conn_1_3}, there exists a sequence $\alpha_1=\beta_1,\beta_2,\ldots,\beta_n=\alpha_2$ of HBCs in $R$ cutting off from $R$ a holed sphere containing the component of $\partial R$ corresponding to $\alpha$. 
The pair $\{\alpha,\beta_k\}$, denoted by $b_k$, is an HBP in $X$, and the sequence $a_1=b_1,b_2,\ldots,b_n=a_2$ is thus a path in $\cal{E}$.  
\end{proof}

The following two lemmas give us information on the images of the hexagons and squares constructed in Lemmas \ref{a_2_2-HBP} and \ref{a_2_1-HBP}, respectively, via an automorphism $\phi$ of $\calcp_s(S)$.

\begin{lem}\label{hex_root}
Let $S=S_{g, 2}$ be a surface with $g\geq 2$.
Let $(v_1,\ldots, v_6)$ be a $6$-tuple defining a hexagon in $\calcp_s(S)$ such that
\begin{itemize}
\item $v_1$, $v_3$ and $v_5$ are $2$-HBPs; $v_2$ and $v_4$ are $1$-HBPs; and
\item each of the four edges $\{ v_1, v_2\}$, $\{ v_2, v_3\}$, $\{ v_3, v_4\}$ and $\{ v_4, v_5\}$ is rooted.
\end{itemize}
Then the root curves of the four edges in the second condition are equal.
\end{lem}

\begin{proof}
We note that in general, if $\{ u_1, u_2\}$ and $\{ u_2, u_3\}$ are rooted edges of $\calcp_s(S)$ such that $u_1$ and $u_3$ are $2$-HBPs and $u_2$ is a $1$-HBP, then the root curves of $\{ u_1, u_2\}$ and $\{ u_2, u_3\}$ are equal. 
This fact implies that the root curves of $\{ v_1, v_2\}$ and $\{ v_2, v_3\}$ are equal and denoted by $\alpha$.
Similarly, the root curves of $\{ v_3, v_4\}$ and $\{ v_4, v_5\}$ are equal and denoted by $\beta$.
The equality $i(v_6, \alpha)=i(v_6, \beta)=0$ then holds because we have $\alpha \in v_1$ and $\beta \in v_5$.
If $\alpha\neq \beta$, then the equality $v_3=\{ \alpha, \beta \}$ would hold.
This contradicts $i(v_3, v_6)\neq 0$.
We thus have $\alpha =\beta$.
\end{proof}

\begin{lem}\label{quad_root}
Let $S=S_{g, 2}$ be a surface with $g\geq 2$.
Let $(v_1,\ldots, v_4)$ be a $4$-tuple defining a square in $\calcp_s(S)$ such that
\begin{itemize}
\item each of $v_1,\ldots, v_4$ is an HBP; and
\item each of the four edges $\{ v_1, v_2\}$, $\{ v_2, v_3\}$, $\{ v_3, v_4\}$ and $\{ v_4, v_1\}$ is rooted.
\end{itemize}
Then the root curves of the four edges in the second condition are equal.
\end{lem}

\begin{proof}
Let $\alpha$ and $\beta$ denote the root curves of $\{ v_1, v_2\}$ and $\{ v_4, v_1\}$, respectively.
If $\alpha\neq \beta$, then we would have $v_1=\{\alpha,\beta\}$.
Since $v_2$ and $v_4$ are disjoint from $v_3$, so is $v_1$.
This is a contradiction.
The lemma is obtained by repeating this argument.
\end{proof}

\begin{lem}\label{sep_HBP_rooted}
Let $S=S_{g,2}$ be a surface with $g \geq 2$, and let $\phi$ be an automorphism of $\calcp_s(S)$.
For each $k=1,2,3,4$, let $a_k=\{\alpha, \alpha_k\}$ be a separating HBP in $S$ such that $\{ a_1, a_2\}$ and $\{ a_3, a_4\}$ are edges of $\calcp_s(S)$. 
Then the root curves of the two edges $\{ \phi(a_1), \phi(a_2)\}$ and $\{ \phi(a_3), \phi(a_4)\}$ of $\calcp_s(S)$ are equal. 
\end{lem}

\begin{proof}
Using Lemma \ref{conn_sep_2}, one can find a path in $\calcp_s(S)$, $a_2=b_2,b_3,\ldots,b_{m-1}=a_3$, such that each $b_k$ is an HBP containing $\alpha$.
We put $b_1=a_1$ and $b_m=a_4$.
For each $k=1,\ldots, m-2$, if $i(b_k,b_{k+2})=0$, then at least two of $b_k$, $b_{k+1}$ and $b_{k+2}$ are equal.
If $i(b_k,b_{k+2}) \ne 0$, we apply either Lemma \ref{a_2_2-HBP} or Lemma \ref{a_2_1-HBP} to $b_k$, $b_{k+1}$ and $b_{k+2}$. 
We then obtain a sequence $\Pi_1,\Pi_2,\ldots,\Pi_n$ of hexagons or squares in $\calcp_s(S)$ such that we have $a_1, a_2 \in \Pi_1$ and $a_3, a_4 \in \Pi_n$; and for each $k$, $\Pi_k$ and $\Pi_{k+1}$ share at least two HBPs.
It follows from Lemmas \ref{hex_root} and \ref{quad_root} that for each $k$, there exists a curve shared by all HBPs in $\phi(\Pi_k)$.
Since $\phi(\Pi_k)$ and $\phi(\Pi_{k+1})$ share at least two HBPs, the curve shared by all HBPs of $\phi(\Pi_k)$ is equal to that of $\phi(\Pi_{k+1})$.
The root curve of $\{ \phi(a_1), \phi(a_2)\}$ is therefore equal to that of $\{ \phi(a_3), \phi(a_4)\}$.
\end{proof}


\subsection{Definition of $\Phi$}\label{subsec-defn}

Let $S=S_{g, 2}$ be a surface with $g\geq 2$, and let $\phi$ be an automorphism of $\calcp(S)$.
We define a bijection $\Phi$ from $V(S)$ onto itself as follows.
If $\alpha$ is an HBC in $S$, then we set $\Phi(\alpha)=\phi(\alpha)$.

If $\beta$ is a non-separating curve in $S$, then choose disjoint and distinct curves $\beta_1$, $\beta_2$ in $S$ such that $\{\beta, \beta_1\}$ and $\{\beta, \beta_2\}$ are both HBPs in $S$, and define $\Phi(\beta)$ to be the root curve of the edge in $\calcp(S)$ consisting of $\phi(\{\beta, \beta_1\})$ and $\phi(\{\beta, \beta_2\})$.
This is well-defined by Lemma \ref{lem-pen2}.

In a similar way, if $\gamma$ is a separating curve in $S$ which is not an HBC in $S$, then choose disjoint and distinct curves $\gamma_1$, $\gamma_2$ in $S$ such that $\{\gamma, \gamma_1\}$ and $\{\gamma, \gamma_2\}$ are both HBPs in $S$, and define $\Phi(\gamma)$ to be the root curve of the edge in $\calcp(S)$ consisting of $\phi(\{\gamma, \gamma_1\})$ and $\phi(\{\gamma, \gamma_2\})$.
Lemma \ref{sep_HBP_rooted} shows that this is well-defined since $\phi$ induces an automorphism of $\calcp_s(S)$ by Lemmas \ref{HBC_to_HBC} and \ref{lem-phi-ns-s}.

We thus obtain a map $\Phi \colon V(S)\rightarrow V(S)$.
Considering $\phi^{-1}$, we see that $\Phi$ is a bijection.
Note that if $\psi$ is an automorphism of $\calcp_s(S)$, then we obtain a bijection $\Psi \colon V_s(S)\rightarrow V_s(S)$ by applying the argument in the previous paragraph.


\section{Construction of $\Phi$ in the case $p \geq 3$}\label{sec-const-p3}

Let $S=S_{g,p}$ be a surface with $g \geq 2$ and $p \geq 3$. 
For an automorphism $\phi$ of $\calcp(S)$, we define a map $\Phi \colon V(S) \rightarrow V(S)$ as in Section \ref{subsec-plan}.
Most of this section is devoted to showing that $\Phi$ is well-defined as in the previous section.
To define $\Phi(\alpha)\in V(S)$ for each curve $\alpha$ in $S$ which is not an HBC in $S$, choosing any two edges $e_1$, $e_2$ of $\calcp(S)$ consisting of two HBPs containing $\alpha$, we have to show that the root curves of $\phi(e_1)$ and $\phi(e_2)$ are equal.
To show it, we connect $e_1$ and $e_2$ by a sequence of rooted $2$-simplices of $\calcp(S)$ whose root curve is equal to $\alpha$.
Note that rooted $2$-simplices of $\calcp(S)$ exist thanks to the assumption $p\geq 3$.

\begin{lem}\label{conn_non-sep}
Let $S=S_{g,p}$ be a surface with $g\geq 2$ and $p\geq 3$. 
For each $k=1,2,3$, let $a_k=\{\alpha, \alpha_k\}$  be a non-separating HBP in $S$ such that $\{ a_1, a_2\}$ and $\{ a_2, a_3\}$ are edges of $\calcp(S)$.
Then there exists a sequence of non-separating HBPs in $S$, $a_1=b_1, b_2,\ldots,b_n=a_3$, such that for each $k=1,2,\ldots,n-1$, the set $\{ a_2, b_k, b_{k+1}\}$ is a rooted $2$-simplex of $\calcp(S)$ whose root curve is equal to $\alpha$.
\end{lem}

To prove this lemma, we use the following:

\begin{prop}[\ci{Proposition 4.4}{kida-tor}]\label{prop-44}
Let $X=S_{0, p}$ be a surface with $p\geq 5$, and choose two distinct components $\partial_1$, $\partial_2$ of $\partial X$.
Then the full subcomplex of $\calc(X)$ spanned by all vertices that correspond to curves in $X$ separating $\partial_1$ and $\partial_2$ is connected.
\end{prop}

\begin{proof}[Proof of Lemma \ref{conn_non-sep}]
If $i(a_1, a_3)=0$, then the lemma obviously holds. 
We assume $i(a_1, a_3) \neq 0$. 
Let $R$ denote the component of $S_{a_2}$ of genus zero, and let $R'$ denote another component of $S_{a_2}$.
Because of $i(a_1, a_2)=i(a_2, a_3)=0$ and $i(a_1, a_3) \ne 0$, either (a) $\alpha_1, \alpha_3 \in V(R)$; or (b) $\alpha_1, \alpha_3 \in V(R')$ occurs.

(a) Suppose $\alpha_1, \alpha_3 \in V(R)$. 
We note that $R$ contains at least two components of $\partial S$.
If $R$ contains at least three components of $\partial S$, then we can find a desired sequence of HBPs in $S$ by using Proposition \ref{prop-44}. 

We now suppose that $R$ contains exactly two components of $\partial S$. 
It then follows that $R'$ contains at least one component of $\partial S$. 
There exists a curve $\beta_2$ in $R'$ with $\{\alpha, \beta_2\}$ an HBP in $S$. 
The sequence $a_1$, $\{\alpha, \beta_2\}$, $a_3$ is a desired one. 

(b) Suppose $\alpha_1, \alpha_3\in V(R')$. 
If $a_2$ is not a $1$-HBP in $S$, then we can find a curve $\beta_2$ in $R$ with $\{\alpha, \beta_2\}$ an HBP in $S$. 
The sequence $a_1$, $\{\alpha, \beta_2\}$, $a_3$ is a desired one. 

If $a_2$ is a $1$-HBP in $S$, then $R'$ contains at least two components of $\partial S$ and thus has at least four boundary components. 
Note that both $\alpha_1$ and $\alpha_3$ are HBCs in $R'$ which separate $\alpha$ and $\alpha_2$ as curves in $R'$. 
By Proposition \ref{connected}, there exists a sequence $\alpha_1=\gamma_1,\gamma_2,\ldots,\gamma_n=\alpha_3$ of HBCs in $R'$ such that each $\gamma_j$ separates $\alpha$ and $\alpha_2$; and any two successive HBCs in that sequence are disjoint and distinct. 
For each $k=1,2, \ldots,n$, the pair $\{\alpha,\gamma_k\}$, denoted by $b_k$, is an HBP in $S$. 
The sequence $a_1=b_1,b_2,\ldots,b_n=a_3$ then satisfies the condition in the lemma.
\end{proof}

The following lemma is an analogue of Lemma \ref{conn_non-sep}, dealing with separating HBPs in $S$ in place of non-separating ones.

\begin{lem}\label{conn_sep}
Let $S=S_{g,p}$ be a surface with $g\geq 2$ and $p\geq 3$. 
For each $k=1,2,3$, let $a_k=\{\alpha, \alpha_k\}$  be a separating HBP in $S$ such that $\{ a_1, a_2\}$ and $\{ a_2, a_3\}$ are edges of $\calcp_s(S)$.
Then there exists a sequence of separating HBPs in $S$, $a_1=b_1, b_2,\ldots,b_n=a_3$, such that for each $k=1,2,\ldots,n-1$, the set $\{ a_2, b_k, b_{k+1}\}$ is a rooted $2$-simplex of $\calcp_s(S)$ whose root curve is equal to $\alpha$.
\end{lem}

\begin{proof}
If $i(a_1, a_3)=0$, then the lemma obviously holds. 
We assume $i(a_1, a_3)\neq 0$. 
We denote by $R$, $R'$ and $R''$ the three components of $S_{a_2}$ so that $R$ is of genus zero, $R'$ and $R''$ are of positive genus; and $R'$ (resp.\ $R''$) has the boundary component corresponding to $\alpha$ (resp.\ $\alpha_2$). 
Note that either (a) $\alpha_1, \alpha_3\in V(R)$; (b) $\alpha_1, \alpha_3 \in V(R')$; or (c) $\alpha_1, \alpha_3 \in V(R'')$ occurs. 

(a) Suppose $\alpha_1, \alpha_3\in V(R)$. 
We note that $R$ contains at least two components of $\partial S$.
If $R$ contains at least three components of $\partial S$, then a desired sequence of HBPs can be obtained as an application of Proposition \ref{prop-44} since as a curve in $R$, each of $\alpha_1$ and $\alpha_3$ separates the two components of $\partial R$ corresponding to $\alpha$ and $\alpha_2$.

If $R$ contains exactly two components of $\partial S$, then $R'$ or $R''$ contains at least one component of $\partial S$. 
Using this component of $\partial S$, we can find a curve $\beta_2$ in $R'$ or $R''$ with $\{\alpha,\beta_2\}$ an HBP in $S$. 
The sequence $a_1$, $\{\alpha,\beta_2\}$, $a_3$ is then a desired one.

(b) Suppose $\alpha_1, \alpha_3\in V(R')$.
If $a_2$ is not a $1$-HBP in $S$, then pick a curve $\beta_2$ in $R$ with $\{\alpha, \beta_2\}$ an HBP in $S$. 
The sequence $a_1$, $\{\alpha,\beta_2\}$, $a_3$ is a desired one. 

Assume that $a_2$ is a $1$-HBP in $S$. 
If $R''$ contains at least one component of $\partial S$, then there exists a curve $\gamma_2$ in $R''$ with $\{\alpha, \gamma_2\}$ an HBP in $S$. 
The sequence $a_1$, $\{\alpha, \gamma_2\}$, $a_3$ is then a desired one. 
If $R''$ contains no component of $\partial S$, then $R'$ contains at least two components of $\partial S$ and thus has at least three boundary components. 
Note that $\alpha_1$ and $\alpha_3$ are HBCs in $R'$ cutting off from $R'$ a holed sphere containing $\alpha$ as a boundary component. 
It follows from Proposition \ref{conn_1_3} that there exists a sequence $\alpha_1=\delta_1,\delta_2,\ldots,\delta_n=\alpha_3$ of HBCs in $R'$ such that each $\delta_k$ cuts off from $R'$ a holed sphere containing $\alpha$ as a boundary component; and any two successive HBCs in that sequence are disjoint and distinct. 
Since for each $k=1, 2,\ldots, n$, the pair $\{\alpha,\delta_k\}$, denoted by $b_k$, is an HBP in $S$, the sequence $a_1=b_1,b_2,\ldots,b_n=a_3$ satisfies the condition of the lemma.

(c) If $\alpha_1, \alpha_3\in V(R'')$, then we can find a desired sequence of HBPs in essentially the same way as in case (b).
\end{proof}

Let $S=S_{g, p}$ be a surface with $g \geq 2$ and $p \geq 3$, and let $\phi$ be an automorphism of $\calcp(S)$.
We define a map $\Phi \colon V(S)\rightarrow V(S)$ in the same way as in Section \ref{subsec-defn}.
This is well-defined thanks to the following:

\begin{lem}\label{root_sep_HBP}
Let $S=S_{g, p}$ be a surface with $g \geq 2$ and $p \geq 3$.
Let $\phi$ be an automorphism of $\calcp(S)$, and for each $k=1,2,3,4$, let $a_k=\{\alpha, \alpha_k\}$ be an HBP in $S$ such that $\{ a_1, a_2\}$ and $\{ a_3, a_4\}$ are edges of $\calcp(S)$.
Then the root curves of the two edges $\{ \phi(a_1), \phi(a_2)\}$ and $\{ \phi(a_3), \phi(a_4)\}$ of $\calcp(S)$ are equal.
\end{lem}

\begin{proof}
By Lemmas \ref{non-sep_conn} and \ref{conn_sep_2}, there exists a sequence of HBPs in $S$, $a_1=b_1, a_2=b_2, b_3,\ldots, b_{n-1}=a_3, b_n=a_4$, such that for each $k$, we have $i(b_k, b_{k+1})=0$, $b_k\neq b_{k+1}$ and $\alpha \in b_k$. 
Applying Lemmas \ref{conn_non-sep} and \ref{conn_sep} to the three HBPs $b_{k-1}$, $b_k$ and $b_{k+1}$ for each $k=2,\ldots, n-1$, we can find a sequence of HBPs in $S$, $b_{k-1}=b_k^1, b_k^2,\ldots, b_k^{m_k}=b_{k+1}$, such that for each $l=1,\ldots, m_k-1$, the set $\{ b_k, b_k^l, b_k^{l+1}\}$, denoted by $\sigma_k^l$, is a rooted $2$-simplex of $\calcp(S)$ whose root curve is equal to $\alpha$.
For any $k$ and $l$, each of $\sigma_k^l\cap \sigma_k^{l+1}$ and $\sigma_k^{m_k-1}\cap \sigma_{k+1}^1$ contains at least two HBPs. 
It turns out that the root curve of the edge $\{ \phi(a_1), \phi(a_2)\}$ is equal to that of the $2$-simplex $\phi(\sigma_k^l)$ for any $k$ and $l$ and is thus equal to that of the edge $\{ \phi(a_3), \phi(a_4)\}$.
\end{proof}

Exchanging symbols appropriately in the above proof, we obtain the following lemma. For any automorphism $\psi$ of $\calcp_s(S)$, if we define a map $\Psi \colon V_s(S)\rightarrow V_s(S)$ in the same way as in Section \ref{subsec-defn}, then the lemma shows that $\Psi$ is well-defined.

\begin{lem}
Let $S=S_{g, p}$ be a surface with $g \geq 2$ and $p \geq 3$.
Let $\psi$ be an automorphism of $\calcp_s(S)$, and for each $k=1,2,3,4$, let $b_k=\{\beta, \beta_k\}$ be a separating HBP in $S$ such that $\{ b_1, b_2\}$ and $\{ b_3, b_4\}$ are edges of $\calcp_s(S)$.
Then the root curves of the two edges $\{ \psi(b_1), \psi(b_2)\}$ and $\{ \psi(b_3), \psi(b_4)\}$ of $\calcp_s(S)$ are equal.
\end{lem}


\section{Simpliciality of $\Phi$}\label{supseinj_Phi}

Let $S=S_{g, p}$ be a surface with $g\geq 2$ and $p\geq 2$.
We show that the bijection $\Phi$ from $V(S)$ onto itself associated to an automorphism of $\calcp(S)$, defined in Sections \ref{sec-const} and \ref{sec-const-p3}, induces an automorphism of $\calc(S)$.

\begin{thm}\label{thm-phi-g2}
Let $S=S_{g,p}$ be a surface with $g \geq 2$ and $p\geq 2$, and let $\phi$ be an automorphism of $\calcp(S)$. 
Then the bijection $\Phi$ from $V(S)$ onto itself associated to $\phi$ induces an automorphism of $\calc(S)$.
\end{thm}

\begin{proof}
It suffices to show that $\Phi$ is simplicial.
Note that by the definition of $\Phi$ and Lemma \ref{root_curve}, we have $\phi(\{ \alpha, \beta \})=\{ \Phi(\alpha), \Phi(\beta)\}$ for each HBP $\{ \alpha, \beta \}$ in $S$.

Let $\alpha$ and $\beta$ be distinct curves in $S$ with $i(\alpha,\beta)=0$. 
We prove the equality $i(\Phi(\alpha), \Phi(\beta))=0$ in the following three cases: (a) both $\alpha$ and $\beta$ are non-separating in $S$; (b) $\alpha$ is non-separating in $S$ and $\beta$ is separating in $S$; and (c) both $\alpha$ and $\beta$ are separating in $S$. 
 
(a) Suppose that both $\alpha$ and $\beta$ are non-separating in $S$.  
If $\alpha$ and $\beta$ are HBP-equivalent, then $\phi(\{\alpha,\beta\})=\{\Phi(\alpha), \Phi(\beta)\}$ is an HBP in $S$, and thus $i(\Phi(\alpha), \Phi(\beta))=0$.
Suppose that $\alpha$ and $\beta$ are not HBP-equivalent. 
We choose non-separating curves $\alpha'$, $\beta'$ in $S$ such that $a=\{\alpha,\alpha'\}$ and $b=\{\beta,\beta'\}$ are disjoint $1$-HBPs in $S$. 
We then have $i(\phi(a), \phi(b))=0$ and thus $i(\Phi(\alpha), \Phi(\beta))=0$.

(b) Suppose that $\alpha$ is non-separating in $S$ and that $\beta$ is separating in $S$. 
If $\beta$ is an HBC in $S$, then we can choose a non-separating curve $\alpha'$ in $S$ disjoint from $\beta$ such that $a=\{\alpha,\alpha'\}$ is an HBP in $S$. 
We then obtain the equality $i(\phi(a),\phi(\beta))=0$ and thus $i(\Phi(\alpha),\Phi(\beta))=0$. 
 
We assume that $\beta$ is not an HBC in $S$. 
Let $R$ denote the component of $S_{\beta}$ containing $\alpha$, and let $R'$ denote another component of $S_{\beta}$. 
If $R$ contains at least one component of $\partial S$, then we choose a curve $\alpha'$ in $R$ and a curve $\beta'$ in $S$ such that $a=\{\alpha,\alpha'\}$ and $b=\{\beta,\beta'\}$ are disjoint HBPs in $S$. 
We then obtain the equality $i(\phi(a),\phi(b))=0$ and thus $i(\Phi(\alpha),\Phi(\beta))=0$.

If $R$ contains no component of $\partial S$, then we choose a curve $\beta'$ in $R'$ with $b=\{\beta,\beta'\}$ a $2$-HBP in $S$.
\begin{figure}
\begin{center}
\includegraphics[width=12cm]{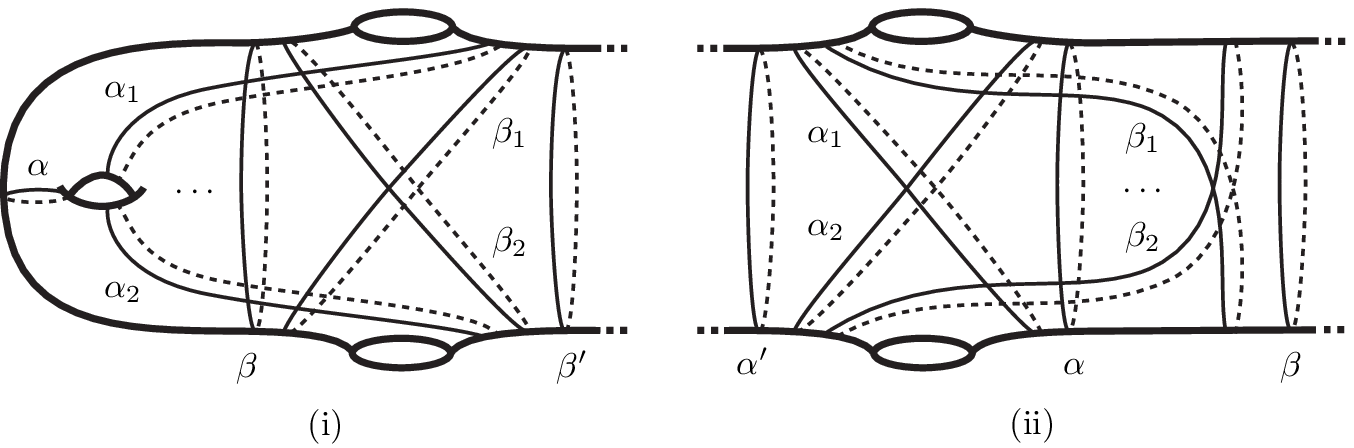}
\caption{}\label{fig-simp}
\end{center}
\end{figure}
As in Figure \ref{fig-simp} (i), we can find curves $\alpha_1$, $\alpha_2$ in $S$ and curves $\beta_1$, $\beta_2$ in $R'$ such that for each $k=1, 2$,
\begin{itemize}
\item each of $a_k=\{\alpha,\alpha_k\}$, $b_k=\{\beta,\beta_k\}$ and $b_k'=\{\beta',\beta_k\}$ is a $1$-HBP in $S$; and  
\item $i(a_k,b_k')=0$ and $i(\beta_1, \beta_2)\neq 0$.
\end{itemize}
Let $Q$ denote the component of $S_{\phi(b)}$ homeomorphic to $S_{0,4}$. 
Since $\phi(b)$ is a $2$-HBP in $S$ and since for each $k=1,2$, both $\phi(b_k)=\{\Phi(\beta),\Phi(\beta_k)\}$ and $\phi(b_k')=\{\Phi(\beta'),\Phi(\beta_k)\}$ are $1$-HBPs in $S$, the curve $\Phi(\beta_k)$ lies in $V(Q)$. 
The two curves $\Phi(\beta_1)$ and $\Phi(\beta_2)$ intersect because $b_1$ and $b_2$ intersect.
Hence, $\Phi(\beta_1)$ and $\Phi(\beta_2)$ fill $Q$. 
Since $\phi(a_k)$ and $\phi(b_k')$ are disjoint for each $k=1,2$, the curve $\Phi(\alpha)$ is disjoint from both $\Phi(\beta_1)$ and $\Phi(\beta_2)$. 
It follows that $\Phi(\alpha)$ is disjoint from $\Phi(\beta)$. 

(c) We assume that both $\alpha$ and $\beta$ are separating in $S$. 
Either if both $\alpha$ and $\beta$ are HBCs in $S$ or if neither $\alpha$ nor $\beta$ is an HBC in $S$ and they are HBP-equivalent, then we have $i(\Phi(\alpha), \Phi(\beta))=0$.

Suppose that $\alpha$ is an HBC in $S$ and that $\beta$ is not an HBC in $S$. 
We choose a curve $\beta'$ disjoint from $\alpha$ such that $\{\beta,\beta'\}$ is an HBP in $S$. 
Since $\phi(\alpha)$ and $\phi(\{\beta,\beta'\})$ are then disjoint, $\Phi(\alpha)$ and $\Phi(\beta)$ are disjoint. 

Finally, we suppose that neither $\alpha$ nor $\beta$ is an HBC in $S$ and they are not HBP-equivalent.
We denote by $R$, $R'$ and $R''$ the three components of $S_{\{\alpha,\beta\}}$ so that $R$ contains $\alpha$ as a boundary component, but does not contain $\beta$ as a boundary component; and $R''$ contains $\alpha$ and $\beta$ as boundary components.
Unless $\partial S$ is contained in either $R$ or $R'$, then there exist two curves $\alpha'$, $\beta'$ in $S$ such that $\{\alpha,\alpha'\}$ and $\{\beta,\beta'\}$ are disjoint HBPs in $S$. 
We thus have $i(\Phi(\alpha),\Phi(\beta))=0$. 

Suppose that $\partial S$ is contained in either $R$ or $R'$. 
Without loss of generality, we may assume that $\partial S$ is contained in $R$. 
Let $\alpha'$ be a curve in $R$ with $a=\{\alpha,\alpha'\}$ a $2$-HBP in $S$. 
As in Figure \ref{fig-simp} (ii), we can find curves $\beta_1$, $\beta_2$ in $S$ and curves $\alpha_1$, $\alpha_2$ in $R$ such that for each $k=1, 2$,
\begin{itemize}
\item each of $a_k=\{\alpha,\alpha_k\}$, $a_k'=\{\alpha',\alpha_k\}$ and $b_k=\{\beta,\beta_k\}$ is a $1$-HBP in $S$; and 
\item $i(a_k', b_k)=0$ and $i(\alpha_1, \alpha_2)\neq 0$. 
\end{itemize} 
We denote by $Q$ the component of $S_{\phi(a)}$ homeomorphic to $S_{0,4}$, which contains $\Phi(\alpha_1)$ and $\Phi(\alpha_2)$ because both $\phi(a_k)=\{\Phi(\alpha),\Phi(\alpha_k)\}$ and $\phi(a_k')=\{\Phi(\alpha'),\Phi(\alpha_k)\}$ are $1$-HBPs in $S$ for each $k=1, 2$.
The curves $\Phi(\alpha_1)$ and $\Phi(\alpha_2)$ fill $Q$ because they intersect.
Since $\Phi(\beta)$ is disjoint from $\Phi(\alpha_1)$ and $\Phi(\alpha_2)$, it is disjoint from $\Phi(\alpha)$. 
\end{proof}

The argument in case (c) in the above proof shows the following:

\begin{thm}\label{thm-aut-bs}
Let $S=S_{g,p}$ be a surface with $g \geq 2$ and $p\geq 2$, and let $\psi$ be an automorphism of $\calcp_s(S)$. 
Then the bijection $\Psi$ from $V_s(S)$ onto itself associated to $\psi$, defined in Sections \ref{sec-const} and \ref{sec-const-p3}, induces an automorphism of $\calc_s(S)$.
\end{thm}

Theorem 1.2 in \cite{kida-tor} shows that if $S=S_{g, p}$ is a surface with $g\geq 2$ and $p\geq 2$, then any automorphism of $\calc_s(S)$ is induced by an element of $\mod^*(S)$. 
Combining Theorem \ref{thm-cc}, we obtain the following:

\begin{cor}\label{cor-aut}
Let $S=S_{g,p}$ be a surface with $g \geq 2$ and $p\geq 2$.
Then any automorphism of $\calcp(S)$ is induced by an element of $\mod^*(S)$.
Moreover, the same conclusion holds for any automorphism of $\calcp_s(S)$.
\end{cor}


\section{Characterization of twisting elements}\label{sec-cha}

The aim of this section is to show that any injective homomorphism from a finite index subgroup of $P_s(S)$ into $P(S)$ induces a superinjective map from $\calcp_s(S)$ into $\calcp(S)$.
This map is obtained by characterizing HBC and HBP twists algebraically. 
Our argument is based on the characterization of such twists due to Irmak-Ivanov-McCarthy \cite{iim}.

\subsection{Preliminaries}

Let $S=S_{g, p}$ be a surface with $g\geq 1$ and assume that the Euler characteristic of $S$, denoted by $\chi(S)=-2g-p+2$, is negative.  
Theorem 3.1 in \cite{iim} asserts that any element $x$ of $P(S)$ is {\it pure} in the following sense (see Theorem A.1 in \cite{kida-tor} for a proof): There exists $\sigma \in \Sigma(S)\cup \{ \emptyset \}$ such that
\begin{enumerate}
\item[(I)] $x$ fixes each curve of $\sigma$ and each component of $S_{\sigma}$ and of $\partial S$; and
\item[(II)] $x$ acts on each component of $S_{\sigma}$ as either the identity or a pseudo-Anosov element.
\end{enumerate}
It follows that $P(S)$ is torsion-free.
For $\tau \in \Sigma(S)$, we denote by $P(S)_{\tau}$ the stabilizer of $\tau$ in $P(S)$.
Theorem A.1 of \cite{kida-tor} shows that for any $x\in P(S)$ and $\tau \in \Sigma(S)$, if $x$ fixes $\tau$, then $x$ fixes each curve of $\tau$ and each component of $S_{\tau}$.
Moreover, $x$ preserves an orientation of each curve of $\tau$.
For each $\tau \in \Sigma(S)$, we thus have the natural homomorphism
\[\theta_{\tau}\colon P(S)_{\tau}\rightarrow \prod_Q\pmod(Q),\]
where $Q$ runs through all components of $S_{\tau}$.
We define $\theta_Q\colon P(S)_{\tau}\rightarrow \pmod(Q)$ as the composition of $\theta_{\tau}$ with the projection onto $\pmod(Q)$.

It is known that for each $x\in P(S)$, there exists the minimal element among all $\sigma \in \Sigma(S)\cup \{ \emptyset \}$ satisfying the above conditions (I) and (II). 
The minimal element is called the {\it canonical reduction system (CRS)} for $x$. 
We note that
\begin{itemize}
\item the CRS's for $x$ and its non-zero power are equal; and
\item if $y\in P(S)$ lies in the centralizer of $x$, then $y$ fixes the CRS for $x$.
\end{itemize}
Indeed, one can also define the CRS for any element and any subgroup of $\mod(S)$. 
If $x\in \mod(S)$ fixes an element of $\Sigma(S)$, then $x$ is said to be {\it reducible}. 
It is known that for each element $x\in \mod(S)$ of infinite order, $x$ is reducible if and only if the CRS for $x$ is non-empty. 
We recommend the reader to consult Section 7 of \cite{iva-subgr} for details on CRS's.

Pick $x\in P(S)$ and let $\sigma \in \Sigma(S)\cup \{ \emptyset \}$ be the CRS for $x$.
We mean by a {\it pA component} for $x$ a component $Q$ of $S_{\sigma}$ on which $x$ acts as a pseudo-Anosov element.

\begin{lem}\label{lem-rank}
Let $S=S_{g, p}$ be a surface with $g\geq 1$ and $\chi(S)<0$. 
Then the maximal rank of finitely generated abelian subgroups of $P(S)$ (resp.\ $P_s(S)$) is equal to $p$ if $g\geq 2$, and equal to $p-1$ if $g=1$.
\end{lem}

\begin{proof}
If $g\geq 2$, then for each simplex $\sigma$ of $\calcp_s(S)$ of maximal dimension, HBP-twists about HBPs in $\sigma$ generate a subgroup of $P_s(S)$ isomorphic to $\mathbb{Z}^p$.
If $g=1$, then there exists a simplex $\tau$ of $\calc_s(S)$ with $|\tau|=p-1$, and Dehn twists about curves in $\tau$ generate a subgroup of $P_s(S)$ isomorphic to $\mathbb{Z}^{p-1}$.
It is thus enough to show that the rank of any finitely generated abelian subgroup of $P(S)$ is at most $p$ if $g\geq 2$, and at most $p-1$ if $g=1$.
We prove it by induction of $p$.

If either $g\geq 2$ and $p=0$ or $g=1$ and $p=1$, then $P(S)$ is trivial, and the lemma is obvious.
We assume either $g\geq 2$ and $p\geq 1$ or $g=1$ and $p\geq 2$.
In general, for any short exact sequence of groups, $1\rightarrow A\rightarrow B\rightarrow C\rightarrow 1$, the inequality ${\rm rk}B\leq {\rm rk}A+{\rm rk}C$ holds, where for a group $D$, we denote by ${\rm rk}D$ the supremum of the ranks of finitely generated abelian subgroups in $D$. 
Fix a component $\partial$ of $\partial S$, and let $R$ be the surface obtained from $S$ by attaching a disk to $\partial$.
Restricting the associated Birman exact sequence
\[1\rightarrow \pi_1(R)\stackrel{\imath}{\rightarrow}\pmod(S)\rightarrow \pmod(R)\rightarrow 1\]
to $P(S)$ and to $P_s(S)$, we obtain the exact sequences
\[1\rightarrow \pi_1(R)\rightarrow P(S)\rightarrow P(R)\rightarrow 1,\quad 1\rightarrow N\rightarrow P_s(S)\rightarrow P_s(R)\rightarrow 1,\]
where we put $N=\imath^{-1}(P_s(S))$.
Since we have ${\rm rk}\pi_1(R)={\rm rk}N=1$, the induction is completed.
\end{proof}


\subsection{Characterization}

We start with the following observation.

\begin{lem}\label{lem-trivial-comp}
Let $S$ be a surface of genus at least one with $\chi(S)<0$.
Pick $x\in P(S)$ and $\sigma \in \Sigma(S)$ such that $x$ fixes $\sigma$.
Let $Q$ be a component of $S_{\sigma}$, and let $\tau$ be the set of all curves of $\sigma$ corresponding to a component of $\partial Q$.
Then the following assertions hold:
\begin{enumerate}
\item If we have $Q\cap \partial S=\emptyset$ and no curve of $\tau$ is an HBC in $S$, then $\theta_Q(x)$ is neutral.
\item If $\sigma$ is the CRS for $x$ and $Q$ is a pA component for $x$, then there exists $y\in P(S)$ such that $\tau$ is the CRS for $y$, $Q$ is a unique pA component for $y$, and the equality $\theta_Q(x)=\theta_Q(y)$ holds. 
\end{enumerate}
\end{lem}

\begin{proof}
Let $\bar{S}$ denote the closed surface obtained from $S$ by attaching disks to all components of $\partial S$.
Let $\iota \colon \pmod(S)\rightarrow \mod(\bar{S})$ be the homomorphism associated with the inclusion of $S$ into $\bar{S}$.
On the assumption in assertion (i), any curve of $\tau$ is essential in $\bar{S}$ (although some of them may be isotopic in $\bar{S}$).
Assertion (i) follows because $\iota(x)$ is neutral.

We prove assertion (ii).
Let $F$ and $C$ be representatives of $x$ and $\tau$, respectively, such that $F(C)=C$ and $F$ is the identity on $C\cup \partial S$.
Let $Q^{\circ}$ denote the component of $S\setminus C$ corresponding to $Q$.
Let $R$ denote the surface obtained from $Q$ by attaching disks to all components of $\partial Q$ corresponding to either a component of $\partial S$ or an HBC in $S$.
The element of $\pmod(R)$ induced by $\theta_Q(x)$ is neutral because $\iota(x)$ is neutral.

We define $G$ as the homeomorphism of $S$ obtained by extending the restriction of $F$ to $Q^{\circ}$ so that $G$ is the identity on $S\setminus Q^{\circ}$.
Let $y\in \pmod(S)$ be the isotopy class of $G$.
The equality $\theta_Q(x)=\theta_Q(y)$ then holds. 
Let $\tau'$ be the set of all curves of $\tau$ that are not HBCs in $S$.
Let $\pi \colon \calc(S)\rightarrow \calc^*(\bar{S})$ be the simplicial map associated with the inclusion of $S$ into $\bar{S}$ (see Section \ref{subsec-complex}).
We put $\bar{\tau}=\pi(\tau')$, which is an element of $\Sigma(\bar{S})\cup \{ \emptyset \}$.
The element $\iota(y)$ fixes $\bar{\tau}$ and acts on each component of $\bar{S}_{\bar{\tau}}$ as the identity.
Multiplying $y$ with appropriate powers of Dehn twists about curves in $\tau'$, we may assume that $\iota(y)$ is neutral.
We then have $y\in P(S)$ and obtain assertion (ii).
\end{proof}

Following terminology in \cite{iim}, we say that an element $x$ of $P(S)$ is {\it basic} if the center of the centralizer of $x$ in $P(S)$ is isomorphic to $\mathbb{Z}$.
The following two propositions characterize basic elements and are stated in Proposition 5.1 of \cite{iim}.

\begin{prop}\label{prop-cha}
Let $S$ be a surface of genus at least one with $\chi(S)<0$.
Let $\Gamma$ be a finite index subgroup of $P(S)$.
Pick $x\in \Gamma$ and let $\sigma \in \Sigma(S)\cup \{ \emptyset \}$ be the CRS for $x$.
We assume that $x$ satisfies one of the following three conditions:
\begin{enumerate}
\item[(a)] There exists a unique pA component $Q$ for $x$, and any curve of $\sigma$ corresponds to a component of $\partial Q$.
Moreover, no curve of $\sigma$ is an HBC in $S$, and no two curves of $\sigma$ form an HBP in $S$.
\item[(b)] $x$ is a non-zero power of an HBC twist.
\item[(c)] $x$ is a non-zero power of an HBP twist.
\end{enumerate}
Then the center of the centralizer of $x$ in $\Gamma$ is isomorphic to $\mathbb{Z}$.
In particular, $x$ is basic.
\end{prop}

\begin{proof}
We denote by $Z(x)$ the centralizer of $x$ in $\Gamma$ and denote by $Z$ the center of $Z(x)$. 
We assume condition (a). 
By Lemma \ref{lem-multitwist}, it is enough to show that for each component $R$ of $S_{\sigma}$ other than $Q$, the group $\theta_R(Z(x))$ either is trivial or contains a pair of independent pseudo-Anosov elements of $\pmod(R)$, where two pseudo-Anosov elements are said to be {\it independent} if they do not generate a virtually cyclic group. 
If $R$ contains a component of $\partial S$ and is not a pair of pants, then $\theta_R(Z(x))$ contains independent pseudo-Anosov elements. 
If $R$ contains no component of $\partial S$, then $Z(x)$ acts on $R$ trivially by Lemma \ref{lem-trivial-comp} (i).

We assume condition (b). 
In this case, $\sigma$ consists of a single HBC $\gamma$ in $S$ and $x$ is a non-zero power of $t_{\gamma}$. 
Let $R$ be a component of $S_{\gamma}$. 
If $R$ contains a component of $\partial S$ and is not a pair of pants, then $\theta_R(Z(x))$ contains independent pseudo-Anosov elements. 
If $R$ contains no component of $\partial S$, then $\gamma$ is a $p$-HBC in $S$, and $R$ is the component of $S_{\gamma}$ of positive genus.
When the genus of $S$ is at least two, the group $\theta_R(Z(x))$ contains independent pseudo-Anosov elements because $R$ contains an HBP in $S$.
When the genus of $S$ is equal to one, $R$ is homeomorphic to $S_{1, 1}$.
The group $\theta_R(Z(x))$ is then trivial because $Z(x)$ acts trivially on the torus obtained by attaching a disk to $\partial R$.
It thus turns out that $Z$ is contained in the cyclic group generated by $t_{\gamma}$.

Finally, we assume condition (c).
In this case, if $x$ is a non-zero power of the element $t_{\alpha}t_{\beta}^{-1}$ with $\{ \alpha, \beta \}$ an HBP in $S$, then we have the equality $\sigma =\{ \alpha, \beta \}$.
We can then apply the same argument as in the case where we assumed condition (a), and we conclude that $Z$ is contained in the cyclic group generated by $t_{\alpha}t_{\beta}^{-1}$.
\end{proof}

\begin{prop}\label{prop-cha-basic}
Let $S$ be a surface of genus at least one with $\chi(S)<0$.
Pick $x\in P(S)$ and let $\sigma \in \Sigma(S)\cup \{ \emptyset \}$ be the CRS for $x$.
If $x$ is basic, then $x$ satisfies one of conditions (a), (b) and (c) in Proposition \ref{prop-cha}.
\end{prop}

\begin{proof}
We first claim that the center of $P(S)$ is trivial.
When the genus of $S$ is at least two, the claim follows from Lemma \ref{lem-faithful} (ii).
When the genus of $S$ is equal to one, any separating curve in $S$ is an HBC in $S$.
Along argument similar to the proof of Lemma \ref{lem-faithful}, we can show that any element of $\mod^*(S)$ fixing any element of $V_s(S)$ fixes any element of $V(S)$ and that any element of $\mod^*(S)$ that commutes any element of $P(S)$ is neutral. 
The claim thus follows.

We denote by $Z(x)$ the centralizer of $x$ in $P(S)$ and denote by $Z$ the center of $Z(x)$.
The claim shown in the last paragraph implies that $x$ is not neutral.
If $x$ is pseudo-Anosov, then condition (a) holds.
We now suppose that $x$ is reducible, and let $\sigma \in \Sigma(S)$ denote the CRS for $x$.
If $\sigma$ contains an HBC $\gamma$ in $S$, then $t_{\gamma}$ lies in $Z$ because $Z(x)$ fixes $\sigma$.
Since $Z$ is isomorphic to $\mathbb{Z}$, $x$ and $t_{\gamma}$ generate a cyclic group. Pureness of $x$ then shows that condition (b) holds.
We can apply the same argument in the case where $\sigma$ contains an HBP in $S$.

We assume that no curve of $\sigma$ is an HBC in $S$ and no two curves of $\sigma$ form an HBP in $S$.
If there were no pA component for $x$, then $x$ would be neutral by Lemma \ref{lem-multitwist}.
This is a contradiction.
If there were two pA components $R_1$, $R_2$ for $x$, then by Lemma \ref{lem-trivial-comp} (ii), we would have $x_1, x_2\in P(S)$ such that 
\begin{itemize}
\item both $x_1$ and $x_2$ fix $\sigma$; and
\item for each $j=1, 2$, $R_j$ is a unique pA component for $x_j$ and we have the equality $\theta_{R_j}(x_j)=\theta_{R_j}(x)$.
\end{itemize}
The two elements $x_1$ and $x_2$ lie in $Z$, and they generate $\mathbb{Z}^2$.
This contradicts the assumption that $Z$ is isomorphic to $\mathbb{Z}$.
It follows that $x$ has a single pA component.
Condition (a) therefore holds.
\end{proof}

\begin{rem}
Let $S$ be a surface of genus at least one with $\chi(S)<0$.
Pick $x\in P(S)$ and denote by $Z$ the center of the centralizer of $x$ in $P(S)$.
We note that $Z$ is finitely generated because any abelian subgroup of $\mod^*(S)$ is finitely generated by Theorem A in \cite{blm}.
Following the proof of Proposition \ref{prop-cha}, we can compute the rank of $Z$ as follows.
Let $\sigma \in \Sigma(S)\cup \{ \emptyset \}$ be the CRS for $x$.
We define $n_1$ to be the number of pA components for $x$ and define $n_2$ to be the number of curves of $\sigma$ which are HBCs in $S$.
Let $\{ \tau_1,\ldots, \tau_m\}$ be the collection of HBP-equivalence classes in $\sigma$.
The rank of $Z$ is then equal to the sum $n_1+n_2+\sum_{i=1}^{m}(|\tau_i|-1)$. 
This result is stated in Proposition 4.1 of \cite{iim}.
\end{rem}

We define $\mathfrak{C}=\mathfrak{C}(S)$ as the set of all non-zero powers of HBC twists in $P(S)$ and define $\mathfrak{P}=\mathfrak{P}(S)$ as the set of all non-zero powers of HBP twists in $P(S)$.

\begin{lem}\label{lem-basic}
Let $S=S_{g, p}$ be a surface with $g\geq 2$.
Let $\Gamma$ be a finite index subgroup of $P_s(S)$, and let $f\colon \Gamma \rightarrow P(S)$ be an injective homomorphism.
If an element $x$ of $\Gamma$ is basic and lie in a finitely generated abelian subgroup of $\Gamma$ of rank $p$, then $f(x)$ is also basic.
In particular, $f(y)$ is basic for each $y\in \mathfrak{C}\cup \mathfrak{P}$.
\end{lem}

To prove this lemma, we use the following general fact, which is essentially verified in Lemma 5.2 in \cite{irmak1} (see Lemma 6.8 in \cite{kida-tor} for a proof).
For a group $A$, we denote by ${\rm rk}A$ the supremum of the ranks of finitely generated abelian subgroups in $A$, and denote by $Z(A)$ the center of $A$.
For each $a\in A$, let $Z_A(a)$ denote the centralizer of $a$ in $A$.

\begin{lem}\label{lem-irmak}
Let $A$ and $B$ be groups with ${\rm rk}A={\rm rk}B<\infty$ and assume that any abelian subgroup of $B$ is finitely generated.
Let $\eta \colon A\rightarrow B$ be an injective homomorphism.
If $a$ is an element of $A$ lying in a finitely generated, free abelian subgroup of $A$ with its rank equal to ${\rm rk}A$, then we have the inequality
\[{\rm rk}Z(Z_B(\eta(a)))\leq {\rm rk}Z(Z_A(a)).\]
\end{lem}

\begin{proof}[Proof of Lemma \ref{lem-basic}]
By Propositions \ref{prop-cha} and \ref{prop-cha-basic}, $Z(Z_{\Gamma}(x))$ is isomorphic to $\mathbb{Z}$.
By Lemmas \ref{lem-rank} and \ref{lem-irmak}, $Z(Z_{P(S)}(f(x)))$ is of rank one.
The lemma thus follows.
\end{proof}

The following lemma characterizes HBP twists among basic elements and is a slight variant of Proposition 6.1 in \cite{iim}.
Let us say that an element $x$ of $P(S)$ is {\it single-pA} if it satisfies condition (a) in Proposition \ref{prop-cha}.

\begin{lem}\label{lem-cha-hbp}
Let $S=S_{g, p}$ be a surface with $g\geq 2$ and $p\geq 2$.
Then the following assertions hold:
\begin{enumerate}
\item For each $x\in \mathfrak{P}\cap P_s(S)$, there exists $y\in \mathfrak{P}\cap P_s(S)$ such that the group generated by $x$ and $y$ is isomorphic to $\mathbb{Z}^2$ and the product $xy$ belongs to $\mathfrak{P}$.
\item Let $z\in P(S)$ be a basic element. If we have a basic element $w\in P(S)$ such that the group generated by $z$ and $w$ is isomorphic to $\mathbb{Z}^2$ and the product $zw$ is basic, then $z$ belongs to $\mathfrak{P}$.
\end{enumerate}
\end{lem}

\begin{proof}
We first prove assertion (i).
Let $\{ \alpha, \beta \}$ be the HBP in $S$ and $k$ the non-zero integer with $x=t_{\alpha}^kt_{\beta}^{-k}$.
Since we have $g\geq 2$ and $p\geq 2$, there exists a curve $\gamma$ in $S$ which is disjoint and distinct from $\alpha$ and $\beta$ and forms an HBP in $S$ with $\beta$ (and thus with $\alpha$).
The element $y=t_{\beta}^kt_{\gamma}^{-k}$ is then a desired one.

We next prove assertion (ii).
Note that $z$ and $w$ fix the CRS's for them because $z$ and $w$ commute.
Since $zw$ is basic, if $z$ were a single-pA element with $Q$ the pA component for $z$, then $w$ would also be a single-pA element with $Q$ the pA component for $w$, and $\theta_Q(z)$ and $\theta_Q(w)$ would generate a virtually cyclic group.
It then follows that $z$ and $w$ generate a virtually cyclic group.
This is a contradiction.
If $z$ were a non-zero power of an HBC twist, then $w$ would be a non-zero power of the same HBC twist since $zw$ is basic.
This also contradicts the assumption that $z$ and $w$ generate $\mathbb{Z}^2$. 
Proposition \ref{prop-cha-basic} then shows $z\in \mathfrak{P}$.
\end{proof}

For each integer $k$ with $1\leq k\leq p$, we denote by $\mathfrak{P}_k=\mathfrak{P}_k(S)$ the subset of $\mathfrak{P}$ consisting of all non-zero powers of HBP twists about $k$-HBPs in $S$.
If $x\in \mathfrak{P}$ is a non-zero power of the HBP twist about an HBP $b$ in $S$, we call $b$ the {\it support} of $x$.
Let us say that two elements $x, y\in \mathfrak{P}$ are {\it equivalent} if the supports of $x$ and $y$ are disjoint and equivalent.

\begin{lem}\label{lem-homo-hbp}
Let $S=S_{g, p}$ be a surface with $g\geq 2$ and $p\geq 2$.
Let $\Gamma$ be a finite index subgroup of $P_s(S)$, and let $f\colon \Gamma \rightarrow P(S)$ be an injective homomorphism.
Then the following assertions hold:
\begin{enumerate}
\item For each $x\in \mathfrak{P}\cap \Gamma$, we have $f(x)\in \mathfrak{P}$.
\item If the supports of two elements $x, y\in \mathfrak{P}\cap \Gamma$ are disjoint and contain a common curve, then the same holds for the supports of $f(x)$ and $f(y)$.
\item If two elements $x, y\in \mathfrak{P}\cap \Gamma$ are equivalent, then $f(x)$ and $f(y)$ are also equivalent.  
\item For each $x\in \mathfrak{P}_p\cap \Gamma$, we have $f(x)\in \mathfrak{P}_p$.
\end{enumerate}
\end{lem}

\begin{proof}
Assertion (i) follows from Lemmas \ref{lem-basic} and \ref{lem-cha-hbp}.
Pick $x, y\in \mathfrak{P}\cap \Gamma$ so that the supports of $x$ and $y$, denoted by $\{ \alpha_1, \alpha_2\}$ and $\{ \beta_1, \beta_2\}$, respectively, are disjoint.
To prove assertion (ii), we first assume $\alpha_1=\beta_1$ and $\alpha_2\neq \beta_2$.
There then exist non-zero integers $j$, $k$ with $x^j y^k\in \mathfrak{P}$.
Since all of $f(x^j)$, $f(y^k)$ and $f(x^j y^k)$ belong to $\mathfrak{P}$ by assertion (i), the supports of $f(x^j)$ and $f(y^k)$ contain a common curve. 
Assertion (ii) is proved.

To prove assertion (iii), we next assume that $\alpha_1$, $\alpha_2$, $\beta_1$ and $\beta_2$ are mutually distinct.
Let $z$ be an HBP twist about the HBP $\{ \alpha_1, \beta_1\}$ and choose non-zero integers $a_1$, $a_2$ and $b$ so that $x^{a_1}z^b$ and $y^{a_2}z^b$ belong to $\mathfrak{P}$ and we have $z^b\in \Gamma$.
Assertion (ii) implies that the supports of $f(x^{a_1})$ and $f(z^b)$ (resp.\ $f(y^{a_2})$ and $f(z^b)$) contain a common curve, and thus $f(x)$ and $f(y)$ are equivalent.

We prove assertion (iv).
It will be shown that for each $y\in (\mathfrak{P}\setminus \mathfrak{P}_p)\cap \Gamma$, we have $f(y)\in \mathfrak{P}\setminus \mathfrak{P}_p$.
Once this is verified, assertion (iv) can be deduced by considering a maximal family of pairwise disjoint and equivalent HBPs in $S$ because such a maximal family contains exactly one $p$-HBP in $S$.

Pick $y\in (\mathfrak{P}\setminus \mathfrak{P}_p)\cap \Gamma$ and let $\{ \gamma, \gamma'\}$ be the support of $y$, which is a separating $k$-HBP in $S$ with $1\leq k\leq p-1$.
There exists a component $Q$ of $S_{\{ \gamma, \gamma'\}}$ such that the genus of $Q$, denoted by $g_1$, is positive and $Q$ contains a component of $\partial S$. 
We may assume that $Q$ contains $\gamma$ as a boundary component.
Choose a separating curve $\delta$ in $Q$ cutting off a surface homeomorphic to $S_{g_1, 1}$ from $Q$.
We then have a curve $\delta'\in V(Q)\cup \{ \gamma \}$ which is separating in $S$, is disjoint from $\delta$ and cuts off from $S$ a surface $R$ homeomorphic to $S_{g_1, 2}$ and containing $\delta$ (see Figure \ref{fig-phbp}).
\begin{figure}
\begin{center}
\includegraphics[width=8cm]{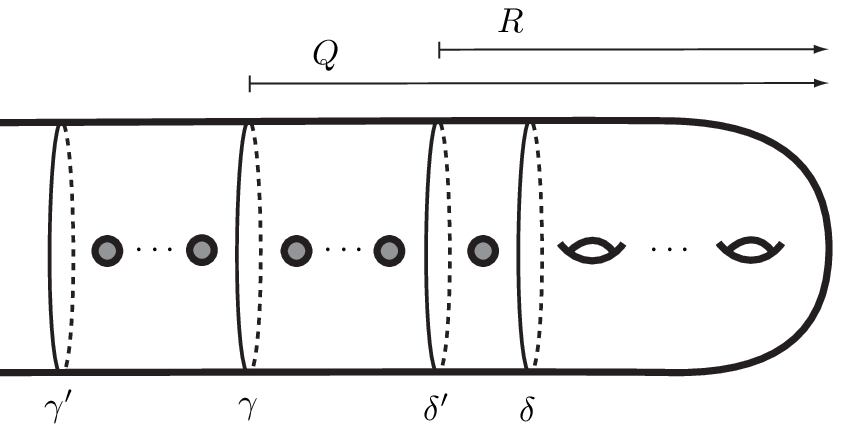}
\caption{}\label{fig-phbp}
\end{center}
\end{figure}
Choose a separating curve $\epsilon$ in $R$ such that $\epsilon$ cuts off a pair of pants from $R$; and $\delta$ and $\epsilon$ fill $R$.
Both $\{ \gamma, \delta \}$ and $\{ \gamma, \epsilon \}$ are then HBPs in $S$.
Assertion (i) implies that $f$ induces a map from the set of separating HBPs in $S$ into the set of HBPs in $S$, which preserves disjointness and non-disjointness.
We denote this map by the same symbol $f$.
We now prove the following:

\begin{claim}
The two HBPs $f(\{ \gamma, \delta \})$ and $f(\{ \gamma, \epsilon \})$ contain a common curve.
\end{claim}

\begin{proof}
Choose a curve $\gamma''$ in $S$ such that $\{ \gamma, \gamma''\}$ is an HBP in $S$ and we have $i(\gamma', \gamma'')\neq 0$. 
We set
\[c'=\{ \gamma, \gamma'\},\quad c''=\{ \gamma, \gamma''\},\quad d=\{ \gamma, \delta\},\quad e=\{ \gamma, \epsilon\}.\]
Note that by assertion (ii), the two HBPs $f(c')$ and $f(d)$ (resp.\ $f(c')$ and $f(e)$) contain a common curve. 
If the claim were not true, then we would have the inclusion $f(c')\subset f(d)\cup f(e)$.
Since $c''$ is disjoint from $d$ and $e$, $f(c'')$ is disjoint from $f(c')$.
This contradicts $i(\gamma', \gamma'')\neq 0$.
\end{proof}

We put $f(\{ \gamma, \delta \})=\{ \gamma_1, \delta_1\}$ and $f(\{ \gamma, \epsilon \})=\{ \gamma_1, \epsilon_1\}$. 
Since $\delta$ and $\epsilon$ intersect, so do $\delta_1$ and $\epsilon_1$.
Fix a non-zero integer $l$ with $t_{\gamma}^lt_{\delta}^{-l}, t_{\gamma}^lt_{\epsilon}^{-l}\in \Gamma$, and let $m$ and $n$ be the non-zero integers with
\[f(t_{\gamma}^lt_{\delta}^{-l})=t_{\gamma_1}^{m}t_{\delta_1}^{-m},\quad f(t_{\gamma}^lt_{\epsilon}^{-l})=t_{\gamma_1}^{n}t_{\epsilon_1}^{-n}.\]
For each integer $q$, we then have $t_{\delta}^{-lq}t_{\epsilon}^{lq}=(t_{\gamma}^lt_{\delta}^{-l})^{q}(t_{\gamma}^lt_{\epsilon}^{-l})^{-q}\in \Gamma$ and
\[f(t_{\delta}^{-lq}t_{\epsilon}^{lq})=f(t_{\gamma}^{lq}t_{\delta}^{-lq})f(t_{\gamma}^{-lq}t_{\epsilon}^{lq})=t_{\gamma_1}^{mq}t_{\delta_1}^{-mq}t_{\gamma_1}^{-nq}t_{\epsilon_1}^{nq}=t_{\gamma_1}^{mq-nq}t_{\delta_1}^{-mq}t_{\epsilon_1}^{nq}.\]
By Theorem 7 in \cite{thurston} or Expos\'e 13, \S III in \cite{flp}, there exists a non-zero integer $r$ satisfying the following three conditions:
\begin{itemize}
\item $t_{\gamma}^{r}t_{\delta}^{-r}$ and $t_{\gamma}^{r}t_{\epsilon}^{-r}$ belong to $\Gamma$, and thus so does $t_{\delta}^{-r}t_{\epsilon}^{r}$;
\item $t_{\delta}^{-r}t_{\epsilon}^{r}$ acts on $R$ as a pseudo-Anosov element; and
\item $f(t_{\delta}^{-r}t_{\epsilon}^{r})$ acts on the subsurface of $S$ filled by $\delta_1$ and $\epsilon_1$ as a pseudo-Anosov element.
\end{itemize}
The element $t_{\delta}^{-r}t_{\epsilon}^{r}$ is single-pA and lies in a free abelian subgroup of $\Gamma$ of rank $p$. 
By Lemma \ref{lem-basic}, $f(t_{\delta}^{-r}t_{\epsilon}^{r})$ is basic and is thus single-pA.

If $f(y)$ were in $\mathfrak{P}_p$, then by Lemma \ref{lem-trivial-comp} (i), there would exist no single-pA element in $P(S)$ commuting $f(y)$. 
This is a contradiction because $f(t_{\delta}^{-r}t_{\epsilon}^{r})$ commutes $f(y)$.
We thus have $f(y)\in \mathfrak{P}\setminus \mathfrak{P}_p$.
\end{proof}

\begin{lem}\label{lem-homo-hbc}
Let $S$, $\Gamma$ and $f\colon \Gamma \rightarrow P(S)$ be the symbols in Lemma \ref{lem-homo-hbp}. 
Then for each $x\in \mathfrak{C}$, we have $f(x)\in \mathfrak{C}\cup \mathfrak{P}$.
\end{lem}

\begin{proof}
Pick $x\in \mathfrak{C}$, and let $\alpha$ be the HBC in $S$ with $x$ a non-zero power of $t_{\alpha}$. 
There then exists a separating $p$-HBP $b$ in $S$ disjoint from $\alpha$.
Let $y\in \Gamma$ be a non-zero power of the HBP twist about $b$. 
Since $f(y)$ lies in $\mathfrak{P}_p$ by Lemma \ref{lem-homo-hbp} (iv) and since $f(y)$ commutes $f(x)$, the element $f(x)$ is not single-pA.  
\end{proof}

Lemmas \ref{lem-homo-hbp} and \ref{lem-homo-hbc} show that any injective homomorphism from a finite index subgroup of $P_s(S)$ into $P(S)$ preserves powers of HBC and HBP twists. 
The next lemma proves the same conclusion for any injective homomorphism from a finite index subgroup of $P(S)$ into $P(S)$.

\begin{lem}\label{lem-homo-cp}
Let $S=S_{g, p}$ be a surface with $g\geq 2$ and $p\geq 2$. 
Let $\Gamma$ be a finite index subgroup of $P(S)$, and let $f\colon \Gamma \rightarrow P(S)$ be an injective homomorphism. 
Then for each $x\in \mathfrak{C}\cup \mathfrak{P}$, we have $f(x)\in \mathfrak{C}\cup \mathfrak{P}$.
\end{lem}

\begin{proof}
It suffices to prove that if $x\in \mathfrak{P}\cap \Gamma$ is a non-zero power of the HBP twist about a non-separating HBP in $S$, then we have $f(x)\in \mathfrak{P}$. 
Along an argument of the same kind as in the proof of Lemma \ref{lem-cha-hbp} (i), we can find a non-zero integer $k$ and an element $y\in \mathfrak{P}\cap \Gamma$ such that the group generated by $x^k$ and $y$ is isomorphic to $\mathbb{Z}^2$ and the product $x^ky$ belongs to $\mathfrak{P}$. 
Following argument in the proof of Lemma \ref{lem-basic}, we can show that $f$ preserves basic elements. 
It then follows from Lemma \ref{lem-cha-hbp} (ii) that $f(x^k)$ belongs to $\mathfrak{P}$, and thus so does $f(x)$. 
\end{proof}

For each HBC $\alpha$ in $S$, we denote by $T_{\alpha}$ the cyclic group generated by $t_{\alpha}$. 
For each HBP $b=\{ \beta, \gamma \}$ in $S$, we denote by $T_{b}$ the cyclic group generated by $t_{\beta}t_{\gamma}^{-1}$. 
As a consequence of Lemmas \ref{lem-homo-hbp}, \ref{lem-homo-hbc} and \ref{lem-homo-cp}, we obtain the following:

\begin{thm}\label{thm-induce}
Let $S=S_{g, p}$ be a surface with $g\geq 2$ and $p\geq 2$.
Then the following assertions hold:
\begin{enumerate}
\item Let $\Gamma$ be a finite index subgroup of $P(S)$, and let $f\colon \Gamma \rightarrow P(S)$ be an injective homomorphism.
Then there exists a superinjective map $\phi \colon \calcp(S)\rightarrow \calcp(S)$ with $f(T_v\cap \Gamma)<T_{\phi(v)}$ for any vertex $v$ of $\calcp(S)$.
\item Let $\Lambda$ be a finite index subgroup of $P_s(S)$, and let $h\colon \Lambda \rightarrow P(S)$ be an injective homomorphism.
Then there exists a superinjective map $\psi \colon \calcp_s(S)\rightarrow \calcp(S)$ with $h(T_v\cap \Lambda)<T_{\psi(v)}$ for any vertex $v$ of $\calcp_s(S)$.
\end{enumerate}
\end{thm}

Theorem \ref{thm-comm} can be deduced from Corollary \ref{cor-aut} and Theorem \ref{thm-induce} along the argument in Section 3 of \cite{iva-aut}. 
We omit details of this part because an argument of the same kind appears in many works \cite{bm}, \cite{irmak1}, \cite{kida-tor}, \cite{kork-aut}, etc.


\end{document}